\documentclass[final,onefignum,onetabnum,letterpaper]{siamonline250211}
\usepackage[scale=0.8]{geometry} 

\usepackage{soul}
\usepackage{lipsum}
\usepackage{amsfonts}
\usepackage{epstopdf}
\ifpdf
  \DeclareGraphicsExtensions{.eps,.pdf,.png,.jpg}
\else
  \DeclareGraphicsExtensions{.eps}
\fi

\usepackage{datetime}
\newdateformat{monthyeardate}{%
	\monthname[\THEMONTH] \THEDAY, \THEYEAR}

\usepackage{academicons}
\usepackage{xcolor}
\renewcommand{\orcid}[1]{\href{https://orcid.org/#1}{\textcolor[HTML]{A6CE39}{orcid.org/#1}}}

\usepackage{amsmath}
\allowdisplaybreaks
\usepackage{amssymb}
\usepackage{commath}
\usepackage{mathtools}
\usepackage{bbm}
\usepackage{bm}

\usepackage{color}
\usepackage{graphicx}
\usepackage[small]{caption}
\usepackage{subcaption}

\usepackage{relsize}
\usepackage{adjustbox}
\usepackage{algorithm}
\usepackage[noend]{algpseudocode}
\usepackage{booktabs}
\usepackage{tikz}
\usepackage{cancel}
\usepackage{pifont}
\usetikzlibrary{bayesnet} 

\usepackage{enumitem}
\setlist[enumerate]{leftmargin=.5in}
\setlist[itemize]{leftmargin=.5in}

\newsiamremark{remark}{Remark}
\newsiamremark{example}{Example}
\newsiamremark{hypothesis}{Hypothesis}
\crefname{hypothesis}{Hypothesis}{Hypotheses}
\newsiamthm{claim}{Claim}

\headers{Flux differencing for FSBP operators}{H.\,M.\ Haase, A.\,R.\ Winters, and J.\ Glaubitz}

\title{Summation-by-parts operators for general function spaces: Flux differencing, entropy stability, and accuracy\thanks{\monthyeardate\today 
\corresponding{Jan Glaubitz} 
}}
\author{ 
Henry M.\ Haase\thanks{
Department of Mathematics, Link\"oping University, Sweden
(\email{henry.haase@liu.se}, \orcid{0009-0008-4065-0426}; 
\email{andrew.ross.winters@liu.se}, \orcid{0000-0002-5902-1522}; 
\email{jan.glaubitz@liu.se}, \orcid{0000-0002-3434-5563})
} 
\and 
Andrew R.\ Winters\footnotemark[2]
\and
Jan Glaubitz\footnotemark[2]
}

\usepackage{amsopn}

\usepackage{comment}

\definecolor{blue1}{RGB}{0,76,153}
\definecolor{green1}{RGB}{1,109,57}
\definecolor{red1}{RGB}{153,0,0}
\definecolor{orange1}{RGB}{255,127,80}

\usepackage[normalem]{ulem}

\usepackage[normalem]{ulem}

\newcommand{\vol}{\textbf{VOL}}
\newcommand{\fvol}{f^{\text{vol}}}

\DeclareMathOperator{\diag}{diag}

\newcommand{\scp}[2]{\left\langle{#1,\, #2}\right\rangle} 
\renewcommand{\d}{\mathrm{d}} 
\newcommand{\intd}{\, \mathrm{d}}

\newcommand{\R}{\mathbb{R}}

\usepackage{lineno}
\usepackage{comment}

\ifpdf
\hypersetup{
	pdftitle={Flux differencing with FSBP operators},
 	pdfauthor={Haase, Winters, and Glaubitz}
}

\fi

\begin{document}

\maketitle

\begin{abstract}
Summation-by-parts (SBP) operators are a cornerstone of provably stable high-order methods for hyperbolic conservation laws.
For nonlinear problems, the flux-differencing framework of Fisher and Carpenter combines such SBP operators with entropy-conservative two-point fluxes to obtain semi-discretizations that are conservative in the sense of Lax--Wendroff and discretely entropy-conservative or stable.
So far, this framework has been developed exclusively for operators that are exact for polynomials.
Function-space SBP (FSBP) operators lift this restriction by being exact on a general finite-dimensional function space that can be tailored to the problem at hand.
In this work, we extend flux differencing to FSBP operators for general (potentially non-polynomial) function spaces.
We prove that, whenever the operator has a diagonal norm and its function space contains the constants, the resulting semi-discretization is conservative in the sense of Lax--Wendroff, semi-discretely entropy-conservative within a single element, and entropy-stable across element interfaces. 
We also provide the first theoretical accuracy analysis of FSBP operators:
We characterize their design order through an algebraic criterion on the underlying function space, prove that suitable non-polynomial spaces achieve the same design order as polynomial spaces of equal dimension, and show that the flux-differencing semi-discretization retains this order.
Notably, the entropy-conservative two-point flux is reused unchanged; only the underlying function space differs from the polynomial case.
Numerical experiments for the Burgers' equation and the compressible Euler equations in one and two dimensions confirm these properties.
\end{abstract}



\begin{Code}
    \url{https://github.com/hmhaase/paper-2026-fsbp-flux-differencing}
\end{Code}

\begin{DOI}
    Not yet assigned
\end{DOI}

\section{Introduction}
\label{sec:intro}

Nonlinear hyperbolic conservation laws model many processes in the physical sciences.
A defining feature of such equations is the interplay of nonlinear wave steepening and the absence of any inherent dissipation mechanism.
As a result, solutions can develop discontinuities in finite time, even for smooth initial data \cite{riemann1860fortpflanzung}.
Weak solutions are then no longer unique, and an additional admissibility criterion---an entropy condition---is required to single out the physically relevant one \cite{lax1960systems,dafermos2016hyperbolic}.
\emph{Entropy-stable} schemes satisfy a discrete counterpart of this entropy condition \cite{godlewski1996numerical,leveque2002finite,hesthaven2017numerical,oeffner2023approximation} and thereby serve a dual purpose:
They mimic the admissibility criterion at discontinuities, and they equip the discretization with a nonlinear stability estimate---essential because, in the absence of inherent dissipation, numerical errors are transported rather than damped.
Entropy stability has accordingly proven crucial for the robustness of high-order methods \cite{tadmor1987numerical,gassner2016split,chen2020review}, which are otherwise prone to nonlinear instabilities for under-resolved solutions, even when these are smooth.

\subsection*{SBP operators}

SBP operators are among the most successful building blocks for systematically constructing provably stable high-order semi-discretizations of initial-boundary-value problems \cite{svard2014review,delreyfernandez2014review}.
Originally developed for finite difference methods \cite{kreiss1974finite,strand1994summation}, they have since been realized in, or extended to, finite volume \cite{nordstrom2001finite}, continuous Galerkin \cite{hicken2016multidimensional,hicken2020entropy,abgrall2020analysisI}, discontinuous Galerkin \cite{gassner2013skew,carpenter2014entropy,chan2018discretely}, radial basis function \cite{glaubitz2024energy},
flux reconstruction \cite{ranocha2016summation,ranocha2018stability,offner2018stability}, active flux \cite{barsukow2025stability,abgrall2025some}, 
and overset-grid \cite{glaubitz2025towards} methods, among others.
The success of SBP operators rests on the discrete energy method, which is easiest explained for linear problems.
To this end, consider partitioning the spatial domain into elements, each equipped with its own SBP operator.
At the continuous level, integration by parts shows that the energy within an element changes only through the fluxes across its boundary.
The SBP property allows us to mimic integration by parts---and with it this energy balance---at the discrete level:
All volume contributions to the discrete energy balance are moved to the element boundaries, so that no spurious energy can be produced inside an element.
What remains is controlling energy production at the element interfaces and the physical boundary.
This is achieved by enforcing interface and boundary conditions weakly, through simultaneous approximation terms \cite{carpenter1994time,carpenter1999stable} or numerical fluxes \cite{toro2013riemann}, which bound the resulting surface terms by adding suitable dissipation.
The result is a semi-discretization whose energy production is controlled entirely by the prescribed boundary data, with no spurious interior contributions---mirroring the continuous energy estimate.

\subsection*{Entropy stability and flux differencing}

For the numerical treatment of \emph{nonlinear} conservation laws, the step from discrete linear energy stability to discrete nonlinear entropy stability for a high-order scheme rests on two more ingredients---in addition to the SBP property.
The first is the notion of an entropy-conservative two-point flux developed by Tadmor \cite{tadmor1987numerical}, which encodes the entropy conservation relation in a discrete, symmetric form \cite{ranocha2018comparison}.
The second is the \emph{flux-differencing} framework of Fisher and Carpenter and collaborators \cite{fisher2013high,fisher2013discretely}:
They showed that any diagonal-norm SBP operator can be rewritten in a telescopic, subcell finite-volume form, and that replacing the subcell flux with an entropy-conservative two-point flux yields a high-order scheme that is conservative in the sense of Lax--Wendroff and discretely entropy-conservative.
Augmenting the element interfaces with suitable dissipation then renders the scheme entropy stable.
This construction is remarkably systematic and operator independent, since the same two-point flux can be combined with different SBP operators.
It underlies many modern robust high-order methods, including split-form and entropy-stable DG schemes \cite{gassner2013skew,gassner2016split,carpenter2014entropy,chen2017entropy,ranocha2023efficient}, and has been extended from collocated Gauss--Lobatto nodes to Gauss nodes and generalized SBP operators \cite{fernandez2014generalized,chan2018discretely,chan2019efficient,mateo2023flux}; see \cite{chen2020review} for a review.

\subsection*{FSBP operators for general function spaces}

At their core, however, most existing SBP operators are designed to be exact for polynomials up to a certain degree.
This reflects the (sometimes implicit) assumption that polynomials are well-suited to approximate the solution of the underlying problem.
For many problems, this is reasonable, but not for all.
Boundary-layer and highly oscillatory or wave-dominated solutions, for instance, are often better represented by exponential or trigonometric functions than by polynomials.
Motivated by this observation, \cite{glaubitz2023summation} introduced \emph{FSBP operators}, which are designed to be exact for a general finite-dimensional function space $\mathcal{F}$, rather than for polynomials.
They demonstrated that most established results for polynomial-based SBP operators---their mimetic properties, the connection to positive quadrature rules, and strategies for their construction---carry over to this more general setting; sometimes under appropriate modifications or restrictions. Furthermore, tailoring $\mathcal{F}$ to a problem at hand can improve accuracy, particularly for coarse discretizations. 
To this end, the function space $\mathcal{F}$ is typically chosen based on some a priori knowledge of the solution behavior. 
For instance, $\mathcal{F}$ can include exponential functions to handle particularly steep gradients at domain boundaries in boundary-layer problems or (approximate) steady-state solutions.
The FSBP framework has since been developed in several directions, including operators for second derivatives \cite{glaubitz2024summation}, multiple dimensions \cite{glaubitz2023multi}, energy-stable RBF methods \cite{glaubitz2024energy}, optimization-based construction \cite{glaubitz2025optimization,glaubitz2026summation}, overset grids \cite{glaubitz2025towards}, and Gaussian FSBP operators with optimal nodes corresponding to generalized Gaussian quadratures \cite{hale2026summation,bercik2026construction,glaubitz2026gaussian}.

So far, however, the flux-differencing framework---and with it the systematic construction of provably conservative and entropy-stable schemes for \emph{nonlinear} conservation laws---has been developed exclusively for polynomial-based SBP operators.
It is therefore natural to ask whether, and under which conditions, flux differencing carries over to FSBP operators for general function spaces.
Answering this question positively would make the entire entropy-stable flux-differencing toolbox readily available to solution-tailored approximation spaces.
At the same time, although the appeal of FSBP operators lies in improved accuracy through solution-tailored function spaces, a systematic analysis of their accuracy---and, in particular, of their design order---has not yet been provided in the literature.
Addressing both aspects is the goal of the present work.

\subsection*{Our contribution}

We show that the flux-differencing framework extends naturally to FSBP operators that are designed to be exact for a general function space $\mathcal{F}$ under mild restrictions. 
These restrictions---for instance, $1 \in \mathcal{F}$---are automatically satisfied for polynomial SBP operators but need to be stated explicitly for FSBP operators. 
Specifically, our main contributions are the following.
\begin{itemize}
	\item 
	We formulate the flux-differencing semi-discretization for FSBP operators and identify the conditions under which it is conservative in the sense of Lax--Wendroff: namely, a diagonal norm matrix and the inclusion of constants in the function space $\mathcal{F}$, i.e., $1 \in \mathcal{F}$.
	
	\item 
	We verify that, when combined with a consistent entropy-conservative two-point flux, the resulting scheme is semi-discretely entropy-conservative on a single element and entropy-stable across element interfaces.
	The two-point flux is reused unchanged; only the underlying function space $\mathcal{F}$ differs from the polynomial case.
	
	\item
	We provide the first theoretical accuracy analysis of FSBP operators.
	We introduce the notion of design order for FSBP operators, bound their truncation error by the simultaneous approximation power of $\mathcal{F}$, and characterize the design order through an algebraic (Wronskian-based) criterion on the function space.
	In particular, suitable non-polynomial spaces achieve the same design order as polynomial spaces of equal dimension, while tailoring the space can substantially reduce the error constant.
	We also show that the flux-differencing semi-discretization retains the design order of the underlying FSBP operator for smooth solutions.
	
	\item 
	We formulate and implement the flux-differencing semi-discretization in a spatially varying way that allows for the use of a different function space $\mathcal{F}_k$, a different set of nodes $\mathbf{x}_k$, and a different FSBP operator $D_k$ in each element of a multi-element approach. 
	Our codebase is publicly available at \cite{haase2026reproducibility}. 
	It builds upon the Julia package Trixi.jl \cite{ranocha2022adaptive} and extends its functionality from spatially static polynomial SBP to spatially varying FSBP operators. 
	Notably, as a special case, if we restrict the solver to use only polynomial function spaces and SBP operators, we obtain the same results, up to round-off, as the polynomial solvers implemented in Trixi.jl.
	
	\item 
	We demonstrate the theory with numerical experiments for the Burgers' equation and the one- and two-dimensional compressible Euler equations, demonstrating discrete conservation, entropy conservation, and entropy stability. 
	We also exemplify improved accuracy over polynomial-based operators for solutions better represented by non-polynomial function spaces.
	
\end{itemize}

\subsection*{Notation}

The notation we use throughout the paper is summarized in \Cref{app:notation}.

\subsection*{Outline}

In \Cref{sec:FSBP}, we recall the necessary background on hyperbolic conservation laws, entropy stability, entropy-conservative fluxes, and FSBP operators.
In \Cref{sec:fluxDiff}, we develop the flux-differencing framework for FSBP operators and establish its conservation and entropy conservation/stability properties.
In \Cref{sec:accuracy}, we analyze the accuracy and design order of FSBP operators and show that the flux-differencing semi-discretization retains this design order.
\Cref{sec:numerics} presents numerical experiments, and \Cref{sec:summary} concludes with a summary and an outlook on future work.
 
\section{Preliminaries}
\label{sec:FSBP}

We present the necessary preliminaries for the remainder of the paper.
We recall one-dimensional hyperbolic conservation laws and their entropy framework (\Cref{subsec:cons_laws}), the notion of entropy-conservative and entropy-stable two-point fluxes (\Cref{subsec:EC_fluxes}), and FSBP operators for general function spaces (\Cref{subsec:FSBP_ops}).
For simplicity, we restrict the presentation to one space dimension. 
Multi-dimensional FSBP operators can be constructed using a tensor-product approach or following \cite{glaubitz2023multi}.
Furthermore, systems are treated componentwise, unless stated otherwise. 
Curvilinear problems are not discussed, however, and left for future work.

\subsection{Hyperbolic conservation laws and entropy}
\label{subsec:cons_laws}

Consider the hyperbolic conservation law
\begin{equation}\label{eq:cons_law}
	\partial_t \boldsymbol{u} + \partial_x \boldsymbol{f}(\boldsymbol{u}) = 0, \quad x \in (x_L, x_R), \quad t > 0,
\end{equation}
equipped with suitable initial and boundary conditions.
Here, $\boldsymbol{u} = [u_1(x,t), \dots, u_m(x,t)]^T \in \Omega \subseteq \R^m$ is the vector of conserved variables and $\boldsymbol{f} \colon \Omega \to \R^m$ is the (possibly nonlinear) flux.
Solutions of \cref{eq:cons_law} can develop discontinuities in finite time, at which point weak solutions cease to be unique \cite{lax1960systems, dafermos2016hyperbolic, leveque2002finite}.
Uniqueness can be restored in some cases by supplementing \cref{eq:cons_law} with an entropy condition \cite{kruzhkov1970first, harten1983symmetric, godlewski1996numerical}.
To this end, let $(S, F)$ be an entropy pair, that is, a strictly convex entropy $S \colon \Omega \to \R$ together with an entropy flux $F \colon \Omega \to \R$ satisfying the compatibility relation
\begin{equation}\label{eq:entropy_compat}
	F'(\boldsymbol{u}) = \boldsymbol{w}(\boldsymbol{u})^T \boldsymbol{f}'(\boldsymbol{u}), \quad 
	\boldsymbol{w}(\boldsymbol{u}) := S'(\boldsymbol{u})^T .
\end{equation}
The quantity $\boldsymbol{w} = \boldsymbol{w}(\boldsymbol{u})$ is called the vector of \emph{entropy variables}. 
We assume the map $\boldsymbol{u} \mapsto \boldsymbol{w}$ is bijective, so that $\boldsymbol{u}$ can be recovered from $\boldsymbol{w}$.
Smooth solutions of \cref{eq:cons_law} then satisfy the entropy identity $\partial_t S(\boldsymbol{u}) + \partial_x F(\boldsymbol{u}) = 0$, whereas entropy-admissible weak solutions satisfy the entropy inequality
\begin{equation}\label{eq:entropy_ineq}
	\partial_t S(\boldsymbol{u}) + \partial_x F(\boldsymbol{u}) \le 0
\end{equation}
in a distributional sense.
Finally, we introduce the \emph{flux potential}
\begin{equation}\label{eq:flux_potential}
	\psi(\boldsymbol{w}) := \boldsymbol{w}^T \boldsymbol{f}(\boldsymbol{u}(\boldsymbol{w})) - F(\boldsymbol{u}(\boldsymbol{w})),
\end{equation}
which satisfies $\psi'(\boldsymbol{w}) = \boldsymbol{f}(\boldsymbol{u}(\boldsymbol{w}))$ and is central to the construction of entropy-conservative fluxes below.

\subsection{Entropy-conservative and entropy-stable numerical fluxes}
\label{subsec:EC_fluxes}

The semi-discretizations of \cref{eq:cons_law} considered in this work are built from two-point numerical fluxes $\boldsymbol{f}_{\mathrm{num}}(\boldsymbol{u}_l, \boldsymbol{u}_k) \colon \Omega \times \Omega \to \R^m$ depending on a pair of states $\boldsymbol{u}_l, \boldsymbol{u}_k$.
A two-point flux is \emph{consistent} if $\boldsymbol{f}_{\mathrm{num}}(\boldsymbol{u}, \boldsymbol{u}) = \boldsymbol{f}(\boldsymbol{u})$ and \emph{symmetric} if $\boldsymbol{f}_{\mathrm{num}}(\boldsymbol{u}_l, \boldsymbol{u}_k) = \boldsymbol{f}_{\mathrm{num}}(\boldsymbol{u}_k, \boldsymbol{u}_l)$.
Following \cite{tadmor1987numerical, tadmor2003entropy}, a consistent and symmetric two-point flux $\boldsymbol{f}_{\mathrm{ec}}$ is called \emph{entropy-conservative} (EC) if
\begin{equation}\label{eq:EC_condition}
	(\boldsymbol{w}_l - \boldsymbol{w}_k)^T \boldsymbol{f}_{\mathrm{ec}}(\boldsymbol{u}_l, \boldsymbol{u}_k) = \psi_l - \psi_k
\end{equation}
holds for all admissible states, where $\boldsymbol{w}_l := \boldsymbol{w}(\boldsymbol{u}_l)$ and $\psi_l := \psi(\boldsymbol{w}_l)$.
Condition \cref{eq:EC_condition} is the discrete counterpart of the chain rule underlying the entropy identity.
Explicit EC fluxes are available for many systems of interest, including the compressible Euler equations \cite{ismail2009affordable, chandrashekar2013kinetic, ranocha2018comparison}, the shallow water equations \cite{fjordholm2011well, wintermeyer2017entropy}, and the ideal magnetohydrodynamics (MHD) equations \cite{winters2016affordable, derigs2018ideal}.
Moreover, related kinetic energy and entropy preserving (KEEP) Euler fluxes, which satisfy \cref{eq:EC_condition} approximately rather than exactly, were developed in \cite{kuya2018kinetic, shima2021preventing}.

An \emph{entropy-stable} (ES) flux $\boldsymbol{f}_{\mathrm{es}}$ is obtained by adding dissipation proportional to the jump $\boldsymbol{w}_k - \boldsymbol{w}_l$ in the entropy variables \cite{fjordholm2012arbitrarily}:
\begin{equation}\label{eq:ES_flux}
	\boldsymbol{f}_{\mathrm{es}}(\boldsymbol{u}_l, \boldsymbol{u}_k) = \boldsymbol{f}_{\mathrm{ec}}(\boldsymbol{u}_l, \boldsymbol{u}_k) - \tfrac{1}{2}\, \Lambda(\boldsymbol{u}_l, \boldsymbol{u}_k)\,(\boldsymbol{w}_k - \boldsymbol{w}_l),
\end{equation}
with a symmetric positive semidefinite dissipation matrix $\Lambda$.
Together with \cref{eq:EC_condition}, this yields
\begin{equation}\label{eq:ES_inequality}
	(\boldsymbol{w}_k - \boldsymbol{w}_l)^T \boldsymbol{f}_{\mathrm{es}}(\boldsymbol{u}_l, \boldsymbol{u}_k) = (\psi_k - \psi_l) - \tfrac{1}{2}(\boldsymbol{w}_k - \boldsymbol{w}_l)^T \Lambda (\boldsymbol{w}_k - \boldsymbol{w}_l) \le \psi_k - \psi_l ,
\end{equation}
which produces entropy dissipation at element interfaces, as made precise in \Cref{sec:fluxDiff}.

For the accuracy analysis in \Cref{sec:accuracy}, we also recall a local consistency property that holds for any consistent and symmetric two-point flux.

\begin{lemma}[Local consistency of two-point fluxes]\label{lem:EC_derivative}
	Let $\boldsymbol{f}_{\mathrm{num}} \colon \Omega \times \Omega \to \R^m$ be a consistent and symmetric two-point flux that is continuously differentiable in a neighborhood around $(\boldsymbol{u},\boldsymbol{u})$. Then
	\begin{equation}\label{eq:EC_derivative}
		2\,\partial_1 \boldsymbol{f}_{\mathrm{num}}(\boldsymbol{u},\boldsymbol{u}) = 2\,\partial_2 \boldsymbol{f}_{\mathrm{num}}(\boldsymbol{u},\boldsymbol{u}) = \boldsymbol{f}'(\boldsymbol{u}),
	\end{equation}
	where $\partial_1$ and $\partial_2$ denote differentiation with respect to the first and second argument, respectively.
\end{lemma}

\begin{proof}
	Differentiating the consistency relation $\boldsymbol{f}_{\mathrm{num}}(\boldsymbol{u},\boldsymbol{u}) = \boldsymbol{f}(\boldsymbol{u})$ provides us with $\partial_1 \boldsymbol{f}_{\mathrm{num}}(\boldsymbol{u},\boldsymbol{u}) + \partial_2 \boldsymbol{f}_{\mathrm{num}}(\boldsymbol{u},\boldsymbol{u}) = \boldsymbol{f}'(\boldsymbol{u})$. 
	Symmetry, $\boldsymbol{f}_{\mathrm{num}}(\boldsymbol{a},\boldsymbol{b}) = \boldsymbol{f}_{\mathrm{num}}(\boldsymbol{b},\boldsymbol{a})$, yields $\partial_1 \boldsymbol{f}_{\mathrm{num}}(\boldsymbol{u},\boldsymbol{u}) = \partial_2 \boldsymbol{f}_{\mathrm{num}}(\boldsymbol{u},\boldsymbol{u})$. Combining both identities proves \cref{eq:EC_derivative}.
\end{proof}

\subsection{FSBP operators}
\label{subsec:FSBP_ops}

We discretize the spatial domain $[x_L, x_R]$ on a grid $\mathbf{x} = [x_1, \dots, x_N]^T$ with $x_L = x_1 < \dots < x_N = x_R$ that includes the boundary points.
For a function $u: \R \to \R$, we write $\mathbf{u} = [u(x_1), \dots, u(x_N)]^T$ and $\mathbf{u}' = [u'(x_1), \dots, u'(x_N)]^T$ for the vectors of nodal values of $u$ and its derivative.
The defining idea of FSBP operators \cite{glaubitz2023summation} is to replace the polynomial exactness of classical SBP operators by exactness on a prescribed function space $\mathcal{F}$. 

\begin{definition}[FSBP operators]\label{def:FSBP}
	Let $\mathcal{F} \subset C^1([x_L, x_R])$ be a finite-dimensional function space.
	Furthermore, let $R \in \R^{2 \times N}$ be the boundary restriction operator with $R \mathbf{u} = [u(x_L), u(x_R)]^T$, and let $B = \diag(-1, 1) \in \R^{2 \times 2}$ encode the outward unit normals at $x_L$ and $x_R$.
	An operator $D = P^{-1} Q \in \R^{N \times N}$ approximating $\partial_x$ on the grid $\mathbf{x} \in \R^N$ is a \emph{$\mathcal{F}$-exact FSBP operator} if
	\begin{enumerate}
		\item[(i)] $D \mathbf{u} = \mathbf{u}'$ for all $u \in \mathcal{F}$,
		\item[(ii)] $P = \diag(p_1, \dots, p_N)$ with $p_n > 0$ for all $n$, and
		\item[(iii)] $Q + Q^T = R^T B R$.
	\end{enumerate}
\end{definition}


Relation (i) in \cref{def:FSBP} ensures that $D$ accurately approximates the continuous derivative operator $\partial_x$ by requiring $D$ to be exact for all functions in $\mathcal{F}$.
Condition (ii) guarantees that $P$ induces a proper discrete inner product and norm, which are given by $\scp{\mathbf{u}}{\mathbf{v}}_P = \mathbf{u}^T P \mathbf{v}$ and $\| \mathbf{u} \|_P^2 = \mathbf{u}^T P \mathbf{u}$, respectively.
Relation (iii) encodes the SBP property in the projection-based form used in modern entropy-stable schemes \cite{fernandez2014generalized, ranocha2023efficient}. 
For the closed operators considered here, we have $R = [\mathbf{e}_1, \mathbf{e}_N]^T$ with $\mathbf{e}_1 = [1,0,\dots,0]^T$ and $\mathbf{e}_N = [0,\dots,0,1]^T$.
The boundary restriction operator $R$ thus selects the first and last nodal values and $R^T B R = \diag(-1, 0, \dots, 0, 1)$.
It allows us to mimic integration by parts (IBP) at a discrete level.
Recall that IBP reads
\begin{equation}\label{eq:IBP}
 	\int_{x_L}^{x_R} u (\partial_x v) \intd x + \int_{x_L}^{x_R} (\partial_x u) v \intd x
		= uv \big|_{x_L}^{x_R}.
\end{equation}
The discrete version of \cref{eq:IBP}, which follows from (iii) in \cref{def:FSBP}, is
\begin{equation}\label{eq:SBP}
	\mathbf{u}^T P D \mathbf{v} + \mathbf{u}^T D^T P \mathbf{v}
		= (R\mathbf{u})^T B (R\mathbf{v}).
\end{equation}
The right-hand side of \cref{eq:SBP} equals $u(x_R) v(x_R) - u(x_L) v(x_L)$, mimicking the boundary term $uv\big|_{x_L}^{x_R}$ in \cref{eq:IBP}, while the left-hand side approximates the two integrals.

Two further properties are subsequently used: the exact differentiation of constants and the characterization of FSBP operators via quadrature rules.

\begin{lemma}[Differentiation of constants]\label{lem:constants}
	If $1 \in \mathcal{F}$, then $D \mathbf{1} = 0$.
\end{lemma}

\begin{proof}
	Apply property (i) of \cref{def:FSBP} to $u \equiv 1$, whose derivative vanishes identically.
\end{proof}

We will see in \Cref{subsec:fd_cons} that inclusion of constants ensures conservation of flux differencing for FSBP operators: 
it removes the spurious volume term $(D\mathbf{1})^T P \boldsymbol{u}$ that would otherwise create or destroy the conserved quantity \cite{glaubitz2023summation}.
We therefore adopt the standing assumption $1 \in \mathcal{F}$ throughout.

Regarding existence, an FSBP operator on the grid $\mathbf{x}$ exists if and only if there is a positive quadrature rule---with weights given by the diagonal of $P$---that is exact for the derivative product space $(\mathcal{F}\mathcal{F})' = \{ (fg)' \mid f, g \in \mathcal{F} \}$; see \cite{glaubitz2023summation}.
Such quadrature rules can be obtained, for instance, by a least-squares approach on sufficiently many nodes \cite{huybrechs2009stable,glaubitz2020stableQF,glaubitz2021stableCFs,glaubitz2023construction}, optimization procedures \cite{glaubitz2023multi,glaubitz2025optimization,glaubitz2026summation}, or with a minimal number of nodes, through generalized Gauss-type rules \cite{ma1996generalized,bremer2010nonlinear,huybrechs2022computation,hale2026summation}.
\section{Flux differencing for FSBP operators}
\label{sec:fluxDiff}

We now combine the FSBP operators of \Cref{subsec:FSBP_ops} with the flux-differencing framework.
Notably, existing flux-differencing semi-discretizations have been built on polynomial SBP operators \cite{fisher2013high,carpenter2014entropy,gassner2016split}. 
We extend them here to the larger class of FSBP operators, applicable to more general (potentially non-polynomial) approximation spaces.
This will enable us to systematically develop provably conservative and entropy-stable FSBP methods for any hyperbolic conservation law for which a consistent, symmetric EC two-point flux is available---which is the case for most relevant problems.
We will see that the existing results and proofs for polynomial SBP operators carry over to FSBP operators with only mild modifications. 
For completeness, we still present the proofs, but move most of them into \Cref{app:conservation,app:entropy}. 

We first introduce the semi-discretization on a single element (\Cref{subsec:fd_scheme}) and establish its conservation (\Cref{subsec:fd_cons}) and entropy properties (\Cref{subsec:fd_entropy}).
We then couple elements and add interface dissipation (\Cref{subsec:fd_multi}).
The accuracy of FSBP operators and of the resulting flux-differencing schemes is analyzed separately in \Cref{sec:accuracy}.
Unless stated otherwise, $D = P^{-1}Q$ is an FSBP operator with $1 \in \mathcal{F}$, and $\boldsymbol{f}^{\mathrm{vol}}$ is a consistent, symmetric entropy-conservative flux.

\subsection{The flux-differencing semi-discretization}
\label{subsec:fd_scheme}

The defining idea of flux differencing is to replace the action $D \mathbf{f}$ of the derivative operator $D$ on the nodal flux values $\mathbf{f}$ by a two-point operation built from a symmetric \emph{volume flux} $\boldsymbol{f}^{\mathrm{vol}}$ \cite{fisher2013high,carpenter2014entropy,gassner2016split}.
On a single element $[x_L,x_R]$ with grid $\mathbf{x} = [x_1, \dots, x_N]^T$ and $x_L = x_1 < \dots < x_N = x_R$, the flux-differencing semi-discretization reads
\begin{equation}\label{eq:fd_scheme}
	\frac{\d}{\d t} \mathbf{u}
        = - \vol
        + P^{-1} R^T B (R \mathbf{f} - \mathbf{f}^{\mathrm{surf}}), \quad
    \vol_{i} = \sum_{n=1}^{N} 2 D_{i,n}\, \boldsymbol{f}^{\mathrm{vol}}( \boldsymbol{u}_{i}, \boldsymbol{u}_{n} ).
\end{equation}
Notably, we consider the system case of \cref{eq:cons_law}, i.e.\ $\boldsymbol{u} \in \R^m$ with flux $\boldsymbol{f} \colon \R^m \to \R^m$. 
That is, the discrete solution $\mathbf{u}$, the nodal flux array $\mathbf{f} = [\boldsymbol{f}(\boldsymbol{u}_1), \dots, \boldsymbol{f}(\boldsymbol{u}_N)]^T$, and the volume term $\vol$ are arrays in $\R^{N \times m}$ whose $i$-th rows are the state $\boldsymbol{u}_i = \boldsymbol{u}(x_i) \in \R^m$, the flux $\boldsymbol{f}(\boldsymbol{u}_i) \in \R^m$, and $\vol_i \in \R^m$, respectively.
The matrix operators $D$, $P$, $R$, and $B$ act on the nodal (row) index and apply componentwise across the $m$ fields, whereas the two-point volume flux $\boldsymbol{f}^{\mathrm{vol}}$ couples all $m$ components at each pair of nodes.
The array $\mathbf{f}^{\mathrm{surf}}$ collects the \emph{surface fluxes} $\boldsymbol{f}^{\mathrm{surf}}_L, \boldsymbol{f}^{\mathrm{surf}}_R \in \R^m$ at the two boundary nodes, which couple the element to external boundary data---or to its neighbors in a multi-element discretization.
Recall that the volume flux $\boldsymbol{f}^{\mathrm{vol}}$ is a consistent and symmetric EC two-point flux, whereas the ES surface flux $\boldsymbol{f}^{\mathrm{surf}}$ is consistent but may incorporate dissipation.

A few remarks are in order.

\begin{remark}
	Let $1 \in \mathcal{F}$ and $D$ is an $\mathcal{F}$-exact FSBP operators. 
	Then \cref{lem:constants} implies that for the arithmetic-mean volume flux $\boldsymbol{f}^{\mathrm{vol}}(\boldsymbol{u}_i, \boldsymbol{u}_j) = \tfrac{1}{2}\big(\boldsymbol{f}(\boldsymbol{u}_i) + \boldsymbol{f}(\boldsymbol{u}_j)\big)$, the volume term reduces to $\vol_i = (D\mathbf{f})_i$. 
	That is, we recover the standard discretization.
\end{remark}

\begin{remark}[Scalar case]\label{rem:fd_systems}
	Setting $m = 1$ recovers a scalar conservation law: the arrays $\mathbf{u}, \mathbf{f}, \vol \in \R^N$ become nodal vectors and the two-point flux $\boldsymbol{f}^{\mathrm{vol}}$ is scalar, with all results below specializing accordingly.
	More generally, the operators act on the stacked solution as $D \otimes I_m$, $P \otimes I_m$, and so on, so that the SBP algebra carries over unchanged from the scalar setting.
\end{remark}

\subsection{Conservation on a single element}
\label{subsec:fd_cons}

A semi-discretization is said to be conservative in the sense of \emph{Lax--Wendroff} if it can be cast in conservation (flux-difference) form.
That is, if the discrete solution changes only through numerical flux differences across cell interfaces.
By the Lax--Wendroff theorem \cite{lax1960systems}, a consistent scheme of this form that converges under grid refinement converges to a weak solution of \cref{eq:cons_law}.
In particular, it reproduces discontinuities at the correct propagation speed.
A conservation form is therefore a prerequisite for reliable shock-capturing.

The following result shows that the flux-differencing FSBP semi-discretization \cref{eq:fd_scheme} possesses exactly this structure, under the two assumptions that recur throughout this section: a diagonal norm and the inclusion of the constants in $\mathcal{F}$.

\begin{corollary}[Discrete conservation]\label{cor:conservation}
	Let $\mathcal{F}$ be a finite-dimensional function space with $1 \in \mathcal{F}$ and let $D = P^{-1} Q$ be an diagonal-norm $\mathcal{F}$-exact FSBP operator.
	Then, the semi-discretization \cref{eq:fd_scheme} satisfies
	\begin{equation}\label{eq:conservation}
		\frac{\d}{\d t}\, \mathbf{1}^T P \mathbf{u} = \boldsymbol{f}^{\mathrm{surf}}_L - \boldsymbol{f}^{\mathrm{surf}}_R ,
	\end{equation}
	i.e., the discrete mass $\mathbf{1}^T P \mathbf{u} = \sum_i p_i \boldsymbol{u}_i \in \R^m$ changes only through the surface fluxes.
	In particular, the scheme is conservative in the Lax--Wendroff sense.
\end{corollary}

\begin{proof}
	The proof from \cite{fisher2013discretely} for polynomial SBP operators carries over to non-polynomial FSBP operators in a straightforward manner. 
	For completeness, we provide an alternative proof in \cref{app:conservation}. 
\end{proof}

Two ingredients drive \cref{cor:conservation}---which are detailed in the proofs provided in \Cref{app:conservation}:
(i) The diagonal norm lets the contraction with $\mathbf{1}^T P$ act on the volume term entrywise. 
(ii) The inclusion of the constants, $1 \in \mathcal{F}$, gives $D\mathbf{1} = 0$ and makes the volume term telescope into a single difference of boundary fluxes.
The two-point flux enters only through its consistency and symmetry, so discrete conservation holds for \emph{any} admissible function space, with the polynomial case recovered as one instance.

\subsection{Entropy conservation on a single element}
\label{subsec:fd_entropy}

Recall that smooth solutions of \cref{eq:cons_law} obey the entropy identity $\partial_t S(\boldsymbol{u}) + \partial_x F(\boldsymbol{u}) = 0$, so the entropy changes only through the physical boundaries.
A discretely entropy-conservative scheme should mimic this balance.
The following result shows that the flux-differencing FSBP semi-discretization does so whenever the volume flux is entropy-conservative, with the total discrete entropy again changing only through the surface fluxes.

\begin{theorem}[Semi-discrete entropy conservation]\label{thm:entropy_cons}
	Let $D = P^{-1}Q$ be a diagonal-norm $\mathcal{F}$-exact FSBP operator with $1 \in \mathcal{F}$ and let $\boldsymbol{f}^{\mathrm{vol}}$ be an EC flux satisfying \cref{eq:EC_condition}.
	Then, the semi-discretization \cref{eq:fd_scheme} satisfies
	\begin{equation}\label{eq:entropy_balance}
		\frac{\d}{\d t}\, \mathbf{1}^T P\, S(\mathbf{u})
			= \left( \psi_N - \boldsymbol{w}_N^T \boldsymbol{f}^{\mathrm{surf}}_R \right)
				- \left( \psi_1 - \boldsymbol{w}_1^T \boldsymbol{f}^{\mathrm{surf}}_L \right),
	\end{equation}
	where $S(\mathbf{u}) = [S(\boldsymbol{u}_1), \dots, S(\boldsymbol{u}_N)]^T$ and $\boldsymbol{w}_i = \boldsymbol{w}(\boldsymbol{u}_i)$.
	In particular, the discrete entropy $\mathbf{1}^T P\, S(\mathbf{u})$ changes only through the surface fluxes.
	Moreover, if $\boldsymbol{f}^{\mathrm{surf}}_L = \boldsymbol{f}(\boldsymbol{u}_1)$ and $\boldsymbol{f}^{\mathrm{surf}}_R = \boldsymbol{f}(\boldsymbol{u}_N)$, then, since $\psi_i - \boldsymbol{w}_i^T \boldsymbol{f}(\boldsymbol{u}_i) = -F(\boldsymbol{u}_i)$ by \cref{eq:flux_potential}, \cref{eq:entropy_balance} reduces to the discrete entropy identity $\frac{\d}{\d t} \mathbf{1}^T P\, S(\mathbf{u}) = F(\boldsymbol{u}_1) - F(\boldsymbol{u}_N)$.
\end{theorem}

\begin{proof}
	See \cref{app:entropy}.
\end{proof}

The mechanism behind \cref{eq:entropy_balance} is a discrete chain rule (see \cref{lem:vol_entropy} in \cref{app:entropy}): 
Contracted with the entropy variables, the volume term telescopes into the boundary difference $F(\boldsymbol{u}_N) - F(\boldsymbol{u}_1)$ of the entropy flux, mirroring the continuous identity $\boldsymbol{w}^T \partial_x \boldsymbol{f} = \partial_x F$.
This requires the entropy-conservative condition \cref{eq:EC_condition} on the two-point flux together with $1 \in \mathcal{F}$.
Whether the total entropy is conserved or dissipated is then governed entirely by the surface fluxes. 
We build upon this in \Cref{subsec:fd_multi} to couple elements in an entropy-stable manner.

\subsection{Entropy stability and element coupling}
\label{subsec:fd_multi}

We now partition the domain $I = [x_L,x_R]$ into $K$ elements $I_k = [x_{L,k}, x_{R,k}]$, $k=1,\dots,K$, with $x_{L,1} = x_L$ and $x_{R,K} = x_R$. 
The elements are coupled through a shared interface flux.
On each element $I_k$, the solution is approximated by an array $\mathbf{u}_k \in \R^{N_k \times m}$ with rows $\boldsymbol{u}_{k,i} = \boldsymbol{u}(x_{k,i}) \in \R^m$, associated with $N_k$ grid points $\{x_{k,i}\}_{i=1}^{N_k}$.
We again assume that each local grid contains the element endpoints, i.e., $x_{k,L} = x_{k,1} < \dots < x_{k,N_k} = x_{k,R}$.
Let $(D_k, P_k, Q_k, R_k)$ define an $\mathcal{F}_k$-exact FSBP operator as in \cref{def:FSBP} on $I_k$.
Notably, the function space $\mathcal{F}_k$ might vary from element to element, and so can the grid points, their number, and the FSBP operators $(D_k, P_k, Q_k, R_k)$.

The resulting multi-element flux-differencing semi-discretization then reads
\begin{equation}\label{eq:fd_scheme_multi}
	\frac{\d}{\d t} \mathbf{u}_k
        = - \vol_k
        + P_k^{-1} R_k^T B (R_k \mathbf{f}_k - \mathbf{f}_k^{\mathrm{surf}}), \quad
    \vol_{k,i} = \sum_{n=1}^{N_k} 2 (D_k)_{i,n}\, \boldsymbol{f}^{\mathrm{vol}}( \boldsymbol{u}_{k,i}, \boldsymbol{u}_{k,n} ),
\end{equation}
where $\mathbf{f}_k = [\boldsymbol{f}(\boldsymbol{u}_{k,1}), \dots, \boldsymbol{f}(\boldsymbol{u}_{k,N_k})]^T$ is the nodal flux array on $I_k$ and $\mathbf{f}_k^{\mathrm{surf}}$ collects the surface fluxes at the two endpoints of $I_k$.
The surface fluxes now couple the element $I_k$ to its neighbors $I_{k-1}$ and $I_{k+1}$ for $k \in \{2,\dots,K-1\}$ and to the external boundary data for $k \in \{1,K\}$.

At an interface between $I_k$ and $I_{k+1}$, we denote the left state $\boldsymbol{u}_{k,R}$ from $I_k$ as $\boldsymbol{u}_{k+\frac{1}{2}}^-$ and the right state $\boldsymbol{u}_{k+1,L}$ from $I_{k+1}$ as $\boldsymbol{u}_{k+\frac{1}{2}}^+$.
The surface flux of the left and of the right element is the same two-point flux, i.e., $\boldsymbol{f}^{\mathrm{surf}}_{k,R} = \boldsymbol{f}^{\mathrm{surf}}_{k+1,L} = \boldsymbol{f}^{\mathrm{surf}}\big(\boldsymbol{u}_{k+\frac{1}{2}}^-, \boldsymbol{u}_{k+\frac{1}{2}}^+ \big)$, which we denote more briefly as $\boldsymbol{f}^{\mathrm{surf}}_{k+\frac{1}{2}}$.
Summing the single-element balance \cref{eq:entropy_balance} over all elements, the interior interface between $I_k$ and $I_{k+1}$ contributes
\begin{equation}\label{eq:interface_production}
\begin{aligned}
	\sigma_{k+\frac{1}{2}}
		& = \left( \psi_{k+\frac{1}{2}}^- - (\boldsymbol{w}_{k+\frac{1}{2}}^-)^T \boldsymbol{f}^{\mathrm{surf}}_{k+\frac{1}{2}} \right)
			- \left( \psi_{k+\frac{1}{2}}^+ - (\boldsymbol{w}_{k+\frac{1}{2}}^+)^T \boldsymbol{f}^{\mathrm{surf}}_{k+\frac{1}{2}} \right) \\
		& = \left( \psi_{k+\frac{1}{2}}^- - \psi_{k+\frac{1}{2}}^+ \right)
			- \big( \boldsymbol{w}_{k+\frac{1}{2}}^- - \boldsymbol{w}_{k+\frac{1}{2}}^+ \big)^T \boldsymbol{f}^{\mathrm{surf}}_{k+\frac{1}{2}},
\end{aligned}
\end{equation}
with $\boldsymbol{w}_{k+\frac{1}{2}}^{\pm} = \boldsymbol{w}\big( \boldsymbol{u}_{k+\frac{1}{2}}^{\pm} \big)$ and $\psi_{k+\frac{1}{2}}^{\pm} = \psi\big( \boldsymbol{u}_{k+\frac{1}{2}}^{\pm} \big)$.
We are now positioned to show entropy conservation/entropy of the multi-element flux-differencing semi-discretization \cref{eq:fd_scheme_multi}.

\begin{theorem}[Entropy conservation/stability of the coupled scheme]\label{thm:coupled}
	Consider the flux-differencing semi-discretization \cref{eq:fd_scheme} on a partition of the domain into elements, each discretized by a diagonal-norm $\mathcal{F}_k$-exact FSBP operator with $1 \in \mathcal{F}_k$, coupled through a common interface flux $\boldsymbol{f}^{\mathrm{surf}}$.
	\begin{enumerate}
		\item[(i)] If the surface flux $\boldsymbol{f}^{\mathrm{surf}}$ is entropy-conservative, i.e., it satisfies \cref{eq:EC_condition}, then every interior interface satisfies $\sigma_{k+\frac{1}{2}} = 0$.
		In this case, the total entropy $\sum_{k} \mathbf{1}^T P_k\, S(\mathbf{u}_k)$ changes only through the physical boundaries.

		\item[(ii)] If the surface flux $\boldsymbol{f}^{\mathrm{surf}}$ is entropy stable, i.e., of the form \cref{eq:ES_flux}, then every interior interface satisfies $\sigma_{k+\frac{1}{2}} = -\frac{1}{2} \big( \boldsymbol{w}_{k+\frac{1}{2}}^+ - \boldsymbol{w}_{k+\frac{1}{2}}^- \big)^T \Lambda \big( \boldsymbol{w}_{k+\frac{1}{2}}^+ - \boldsymbol{w}_{k+\frac{1}{2}}^- \big) \le 0$.
		In this case, the total entropy is bounded by the physical boundary fluxes.
	\end{enumerate}
\end{theorem}

\begin{proof}
	For (i), the EC condition \cref{eq:EC_condition} applied to $\boldsymbol{f}^{\mathrm{surf}}_{k+\frac{1}{2}}$ at the states $\boldsymbol{u}_{k+\frac{1}{2}}^-, \boldsymbol{u}_{k+\frac{1}{2}}^+$ gives
	\begin{equation}
		\big( \boldsymbol{w}_{k+\frac{1}{2}}^- - \boldsymbol{w}_{k+\frac{1}{2}}^+ \big)^T \boldsymbol{f}^{\mathrm{surf}}_{k+\frac{1}{2}} = \psi_{k+\frac{1}{2}}^- - \psi_{k+\frac{1}{2}}^+,
	\end{equation}
	so \cref{eq:interface_production} vanishes.
	For (ii), inserting the entropy-stable $\boldsymbol{f}^{\mathrm{surf}}$ from \cref{eq:ES_flux} into \cref{eq:interface_production} and using \cref{eq:ES_inequality} with $(l, k) = (-, +)$ yields
	\begin{equation}
		\sigma_{k+\frac{1}{2}}
			= - \frac{1}{2} \big( \boldsymbol{w}_{k+\frac{1}{2}}^+ - \boldsymbol{w}_{k+\frac{1}{2}}^- \big)^T \Lambda \big( \boldsymbol{w}_{k+\frac{1}{2}}^+ - \boldsymbol{w}_{k+\frac{1}{2}}^- \big) \\
			\le 0,
	\end{equation}
	since $\Lambda$ is positive semidefinite.
	Summing over all interfaces and invoking \cref{thm:entropy_cons} on each element gives the claimed bounds.
\end{proof} 
\section{Accuracy of FSBP operators}
\label{sec:accuracy}

We now investigate the accuracy of FSBP operators and of the flux-differencing semi-discretization \cref{eq:fd_scheme}.
While the conservation and entropy-stability results of \Cref{sec:fluxDiff} carry over from the polynomial setting with only mild modifications, the accuracy of FSBP operators has not yet been systematically analyzed in the literature.
We characterize when an FSBP operator has \emph{design order} $p$, link this order to the approximation power of the function space $\mathcal{F}$, and show that the flux-differencing volume term retains it.

\begin{definition}[Design order of FSBP operators]
	We say that an FSBP operator $D$ has \emph{design order} $p$ if $(D \mathbf{g})_i = g'(x_i) + \mathcal{O}(h^p)$ for every sufficiently smooth scalar function $g$ with nodal-value vector $\mathbf{g}$.
	Here, $h = \max_{i = 1, \dots, N-1} | x_{i+1} - x_i |$ is the \emph{grid spacing}, i.e., the largest distance between two neighboring nodes of the grid $x_1 < \dots < x_N$ on which $D$ operates.
	The statement $\mathcal{O}(h^p)$ is understood asymptotically, i.e., for a sequence of grids and associated FSBP operators with grid spacing $h \to 0$.
\end{definition}

By construction, an $\mathcal{F}$-exact FSBP operator differentiates its function space exactly, i.e., $D\mathbf{g} = \mathbf{g}'$ for all $g \in \mathcal{F}$ (\cref{def:FSBP}).
The design order of an $\mathcal{F}$-exact FSBP operator thus reflects how well $\mathcal{F}$ approximates a sufficiently smooth scalar function $g$ on the grid.

Throughout this section, we work in the multi-element setting of \Cref{subsec:fd_multi}: the domain is partitioned into elements $I_k$ of size $h_k \leq h$, and each element carries its own function space $\mathcal{F}_k$, grid, and $\mathcal{F}_k$-exact FSBP operator $D_k$, with a uniformly bounded number of nodes per element.
We study the limit $h \to 0$; the element sizes and the local grid spacings then agree up to constant factors.
We call such a family of FSBP operators \emph{norm-bounded} if $\| D_k \|_\infty \leq C_D h_k^{-1}$ with a constant $C_D$ independent of the element and of $h$.
Here, $\| D \|_\infty = \max_{i = 1, \dots, N} \sum_{j = 1}^{N} | D_{ij} |$ denotes the matrix norm induced by the maximum vector norm.
A family of FSBP operators is norm-bounded, for instance, whenever the operators on all elements are obtained from finitely many fixed operators on a reference element by affine maps.
Since the analysis below is local to a single element, we henceforth consider a fixed but arbitrary element and suppress the element index $k$; all assumptions and estimates are understood to hold uniformly over the elements.
All suprema are taken over the element under consideration, and $\bar{x}$ denotes a fixed expansion point in that element.

Our starting point is the observation that the truncation error of an FSBP operator is controlled by how well $\mathcal{F}$ \emph{simultaneously} approximates a function and its derivative.

\begin{lemma}[Truncation error bound]\label{lem:truncation}
	Let $D$ be an $\mathcal{F}$-exact FSBP operator and let $g \in C^1$ have the nodal-value vector $\mathbf{g}$.
	Then, 
	\begin{equation}\label{eq:truncation_bound}
		\left| (D \mathbf{g})_i - g'(x_i) \right|
			\leq \inf_{f \in \mathcal{F}} \left( \| D \|_\infty \| g - f \|_{L^\infty} + \| g' - f' \|_{L^\infty} \right)
	\end{equation}
	holds for $i = 1, \dots, N$.
\end{lemma}

\begin{proof}
	Let $f \in \mathcal{F}$ with nodal-value vector $\mathbf{f}$.
	Exactness gives $D \mathbf{f} = \mathbf{f}'$, and therefore $(D \mathbf{g})_i - g'(x_i) = (D (\mathbf{g} - \mathbf{f}))_i - (g - f)'(x_i)$.
	The assertion follows from $|(D (\mathbf{g} - \mathbf{f}))_i| \leq \| D \|_\infty \max_j |(g - f)(x_j)| \leq \| D \|_\infty \| g - f \|_{L^\infty}$ and taking the infimum over $f \in \mathcal{F}$.
\end{proof}

For a norm-bounded family of $\mathcal{F}$-exact FSBP operators, \cref{lem:truncation} yields design order $p$ as soon as $\mathcal{F}$ approximates every sufficiently smooth function to $\mathcal{O}(h^{p+1})$ in the maximum norm and its derivative to $\mathcal{O}(h^{p})$.
By Taylor expansion, it suffices to check this on polynomials.

\begin{lemma}[Reduction to polynomials]\label{lem:poly_reduction}
	A norm-bounded family of $\mathcal{F}$-exact FSBP operators has design order $p$ if and only if it differentiates every polynomial of degree at most $p$ with truncation error $\mathcal{O}(h^p)$.
\end{lemma}

\begin{proof}
	Necessity is clear since polynomials are smooth.
	For sufficiency, let $g$ be sufficiently smooth and decompose it as $g = \pi + r$, where $\pi \in \mathcal{P}_p$ is the Taylor polynomial of $g$ of degree $p$ around $\bar{x}$ and the remainder $r$ satisfies $\| r \|_{L^\infty} = \mathcal{O}(h^{p+1})$ and $\| r' \|_{L^\infty} = \mathcal{O}(h^{p})$.
	The truncation error of $\pi$ is $\mathcal{O}(h^p)$ by assumption, since the Taylor coefficients of $g$ are bounded independently of $h$.
	For the remainder, arguing as in the proof of \cref{lem:truncation}, $|(D \mathbf{r})_i - r'(x_i)| \leq \| D \|_\infty \| r \|_{L^\infty} + \| r' \|_{L^\infty} \leq C_D h^{-1} \mathcal{O}(h^{p+1}) + \mathcal{O}(h^p) = \mathcal{O}(h^p)$.
	Combining both estimates by linearity proves the claim.
\end{proof}

It remains to determine when $\mathcal{F}$ reproduces polynomials to the required order.
The following result provides an algebraic criterion in terms of the Taylor polynomials of a basis of $\mathcal{F}$.
It is the FSBP analogue of the criterion of Yuan and Shu \cite{yuan2006discontinuous} for DG methods based on non-polynomial approximation spaces, in the Wronskian form of Yang and Yoon \cite{yang2016short}.

\begin{proposition}[Design order and the function space]\label{prop:design_order}
	Let $\mathcal{F} = \mathrm{span}\{ v_0, \dots, v_k \}$ with $v_0, \dots, v_k \in C^{k+1}$, and suppose that the Wronskian matrix
	\begin{equation}\label{eq:wronskian}
		W(\bar{x}) := \left[ v_{\ell}^{(n)}(\bar{x}) \right]_{n, \ell = 0}^{k} \in \R^{(k+1) \times (k+1)}
	\end{equation}
	is invertible, with $\| W(\bar{x})^{-1} \|_\infty$ and $\max_n \| v_n^{(k+1)} \|_{L^\infty}$ bounded independently of $h$.
	Then the following holds:
	\begin{enumerate}
		\item[(i)] For every $\pi \in \mathcal{P}_k$, there exists an $f \in \mathcal{F}$ with $\| \pi - f \|_{L^\infty} = \mathcal{O}(h^{k+1})$ and $\| \pi' - f' \|_{L^\infty} = \mathcal{O}(h^{k})$.
		\item[(ii)] Every norm-bounded family of $\mathcal{F}$-exact FSBP operators has design order $p = k = \dim \mathcal{F} - 1$---the same design order as for the polynomial space $\mathcal{P}_k$ of equal dimension.
	\end{enumerate}
\end{proposition}

\begin{proof}
	(i) For $n = 0, \dots, k$, let $(T^k v_n)(x) := \sum_{\ell=0}^{k} \frac{v_n^{(\ell)}(\bar{x})}{\ell!} (x - \bar{x})^\ell$ denote the degree-$k$ Taylor polynomial of $v_n$ around $\bar{x}$.
	Invertibility of $W(\bar{x})$ means precisely that $T^k v_0, \dots, T^k v_k$ form a basis of $\mathcal{P}_k$.
	Hence, given $\pi \in \mathcal{P}_k$, we can write $\pi = \sum_{n=0}^{k} c_n\, T^k v_n$ with coefficients $c_n$ that are bounded in terms of $\| W(\bar{x})^{-1} \|_\infty$ and the coefficients of $\pi$, uniformly in $h$.
	For $f := \sum_{n=0}^{k} c_n v_n \in \mathcal{F}$, we obtain $\pi - f = -\sum_{n=0}^{k} c_n ( v_n - T^k v_n )$.
	The assertion now follows from the Taylor remainder bounds $\| v_n - T^k v_n \|_{L^\infty} \leq \frac{h^{k+1}}{(k+1)!} \| v_n^{(k+1)} \|_{L^\infty}$ and $\| (v_n - T^k v_n)' \|_{L^\infty} \leq \frac{h^{k}}{k!} \| v_n^{(k+1)} \|_{L^\infty}$, where the latter uses $(T^k v_n)' = T^{k-1} (v_n')$.
	(ii) Combining (i) with \cref{lem:truncation} shows that every polynomial of degree at most $k$ is differentiated with truncation error $\| D \|_\infty\, \mathcal{O}(h^{k+1}) + \mathcal{O}(h^{k}) = \mathcal{O}(h^{k})$.
	\Cref{lem:poly_reduction} extends this to all sufficiently smooth functions.
\end{proof}

Two comments are in order.
First, the invertibility of the Wronskian matrix \cref{eq:wronskian} is a property of the function space $\mathcal{F}$ (and the point $\bar{x}$), not of the chosen basis, since changing the basis multiplies $W(\bar{x})$ by an invertible matrix.
Second, the criterion is essentially sharp: if the degree-$k$ Taylor polynomials of $\mathcal{F}$ span only a proper subspace of $\mathcal{P}_k$, then some polynomial of degree at most $k$ cannot be approximated from $\mathcal{F}$ to the order required in (i). 
In this case, $\mathcal{F}$-exactness alone no longer guarantees design order $k$; see \cref{expl:design_order_fail} below.
At the same time, the conditions of \cref{prop:design_order} are satisfied by extended Chebyshev systems, whose Wronskians are positive \cite{karlin1966tchebycheff,yang2016short}. 
For an extended Chebyshev system on the whole computational domain, these constants are moreover uniform over all elements, as required in the element-wise varying setting considered here.
This covers, in particular, spaces spanned by exponential and trigonometric polynomials, and thus all function spaces considered in \Cref{sec:numerics}.

\begin{remark}[From nodewise to discrete $L^2$ estimates]\label{rem:discrete_L2}
	All estimates in this section are nodewise and hence control the truncation error in the maximum norm.
	They transfer to the discrete $L^2$ norms used in \Cref{sec:numerics} with explicit constants.
	Indeed, combining the SBP property with $\mathcal{F}$-exactness shows that the quadrature induced by $P$ is exact for $(fg)'$ with $f, g \in \mathcal{F}$. 
	Choosing $f = x$ and $g = 1$ shows that $\sum_{i=1}^{N} p_i$ equals the length of the element.
	Every nodewise $\mathcal{O}(h^p)$ bound therefore implies
	\begin{equation}\label{eq:L2_transfer}
		\| D \mathbf{g} - \mathbf{g}' \|_{P}
			\leq \Big( \sum_{i=1}^{N} p_i \Big)^{1/2} \max_{i = 1, \dots, N} \left| (D \mathbf{g})_i - g'(x_i) \right|
			= \mathcal{O}\big(h^{p + 1/2}\big)
	\end{equation}
	on each element.
	To obtain the corresponding global bound, index the elements by $k$ and suppose that $\max_{i} | (D_k \mathbf{g}_k)_i - g'(x_{k,i}) | \leq C h^p$ holds uniformly over all elements, where $\mathbf{g}_k$ collects the nodal values of $g$ on the element $I_k$.
	Squaring \cref{eq:L2_transfer} on each element and summing over $k$, using that the weights of $P_k$ sum to the element length $h_k$ and that the element lengths sum to the domain length, $\sum_{k} h_k = x_R - x_L$, then yields
	\begin{equation}\label{eq:L2_global}
		\sum_{k} \| D_k \mathbf{g}_k - \mathbf{g}_k' \|_{P_k}^2
			\leq \sum_{k} h_k \left( C h^p \right)^2
			= (x_R - x_L)\, C^2 h^{2p} .
	\end{equation}
	Taking the square root gives the bound $C\, (x_R - x_L)^{1/2}\, h^p$ in the global discrete $L^2$ norm $\big( \sum_{k} \| \cdot \|_{P_k}^2 \big)^{1/2}$. 
	The factor $h^{1/2}$ gained on each element in \cref{eq:L2_transfer} is exactly consumed by the $\mathcal{O}(h^{-1})$ number of elements, so that the order $p$ is preserved.
	As with all statements in this section, these bounds concern truncation errors. 
	Convergence rates for the solution error do not follow from them alone; also see the outlook in \Cref{sec:summary}.
\end{remark}

\begin{example}
\label{expl:design_order_poly}
	For the polynomial space $\mathcal{F} = \mathcal{P}_d$, \cref{prop:design_order} recovers the classical design order $p = d$.
\end{example}

\begin{example}
\label{expl:design_order_exp}
	For the exponential space $\mathcal{E}_d = \mathrm{span}\{1, x, \dots, x^{d-1}, e^{\alpha x}\}$ used in \Cref{sec:numerics}, evaluating \cref{eq:wronskian} in the re-centered basis $\{1, x - \bar{x}, \dots, (x - \bar{x})^{d-1}, e^{\alpha x}\}$ yields a triangular matrix supplemented by the row $e^{\alpha \bar{x}} [1, \alpha, \dots, \alpha^{d}]$, whose determinant is a nonzero multiple of $\alpha^{d} e^{\alpha \bar{x}}$.
	Hence, $\mathcal{E}_d$ yields the design order $p = d$ for every $\alpha \neq 0$---the same as $\mathcal{P}_d$.
\end{example}

\begin{example}
\label{expl:design_order_fail}
	The conditions of \cref{prop:design_order} can also fail: for $\mathcal{F} = \mathrm{span}\{1, x - \bar{x}, \sin(\alpha (x - \bar{x}))\}$, the sine contributes no quadratic Taylor term at $\bar{x}$, so the degree-$2$ Taylor polynomials span only $\mathcal{P}_1$.
	Indeed, every element of $\mathcal{F}$ is the sum of a constant and a function that is odd about $\bar{x}$, so that $(x - \bar{x})^2$ is at distance $\Theta(h^2)$ from $\mathcal{F}$.  
	This three-dimensional space thus supports only design order $p = 1$.
	Replacing $x - \bar{x}$ by $\cos(\alpha (x - \bar{x}))$ restores the full design order $p = 2$.
\end{example}

\begin{remark}[Same order, smaller constants]\label{rem:tailored_accuracy}
	\Cref{prop:design_order} shows that, for generic smooth solutions, tailoring $\mathcal{F}$ does not increase the design order beyond that of a polynomial space of equal dimension.
	The benefit of tailored spaces instead lies in the error constant.
	By \cref{lem:truncation}, the truncation error is proportional to the distance of the differentiated function from $\mathcal{F}$ rather than from $\mathcal{P}_p$; for the volume term below, the relevant function is $\xi \mapsto 2 \boldsymbol{f}^{\mathrm{vol}}(\boldsymbol{u}(x_i), \boldsymbol{u}(\xi))$, which for nonlinear fluxes involves nonlinear compositions of the solution.
	If these are well captured by $\mathcal{F}$, the error constant shrinks; if they lie in $\mathcal{F}$, the volume term is even exact at the nodes.
	Enriching the approximation space with, for instance, (an approximation of) a steady solution thus retains the order of the scheme while substantially reducing the error near that solution.
\end{remark}

Having identified when an FSBP operator has design order $p$, we now show that the flux-differencing volume term inherits this order.

\begin{proposition}[Accuracy of the volume term]\label{prop:accuracy}
	Let $\boldsymbol{u}$ be smooth and let $\boldsymbol{f}^{\mathrm{vol}}$ be a consistent, symmetric two-point flux that is continuously differentiable in a neighborhood of the diagonal.
	If the FSBP operator $D$ has design order $p$, then
	\begin{equation}\label{eq:accuracy}
		\vol_i = \partial_x \boldsymbol{f}(\boldsymbol{u}(x_i)) + \mathcal{O}(h^p) .
	\end{equation}
\end{proposition}

\begin{proof}
	Fix $i$ and define the smooth function $\boldsymbol{g}(\xi) := 2\, \boldsymbol{f}^{\mathrm{vol}}(\boldsymbol{u}(x_i), \boldsymbol{u}(\xi))$, with nodal-value array $\mathbf{g} = [\boldsymbol{g}(x_1), \dots, \boldsymbol{g}(x_N)]^T$, so that $\vol_i = \sum_j D_{ij}\, \boldsymbol{g}(x_j) = (D \mathbf{g})_i$.
	Since $\boldsymbol{g}$ is smooth, the design-order assumption (applied componentwise) gives $(D \mathbf{g})_i = \boldsymbol{g}'(x_i) + \mathcal{O}(h^p)$.
	By the chain rule and \cref{lem:EC_derivative},
	\begin{equation}
		\boldsymbol{g}'(x_i) = 2\, \partial_2 \boldsymbol{f}^{\mathrm{vol}}(\boldsymbol{u}(x_i), \boldsymbol{u}(x_i))\, \boldsymbol{u}'(x_i) = \boldsymbol{f}'(\boldsymbol{u}(x_i))\, \boldsymbol{u}'(x_i) = \partial_x \boldsymbol{f}(\boldsymbol{u}(x_i)) ,
	\end{equation}
	which proves \cref{eq:accuracy}.
\end{proof}

\begin{corollary}[Consistency and order]\label{cor:consistency}
	The flux-differencing semi-discretization \cref{eq:fd_scheme} is consistent with the conservation law \cref{eq:cons_law} and inherits the design order $p$ of the underlying FSBP operator.
\end{corollary}

\begin{proof}
	Immediate from \cref{prop:accuracy}, since the surface term in \cref{eq:fd_scheme} is a consistent boundary correction supported at the boundary nodes.
\end{proof}

Since the analysis of this section is local to a single element, it carries over verbatim to the element-wise varying setting of \Cref{subsec:fd_multi}, where both the function space $\mathcal{F}_k$ and the FSBP operator $D_k$ may change from element to element.
The following result makes this precise.

\begin{corollary}[Element-wise function spaces]\label{cor:varying_spaces}
	Consider the multi-element semi-discretization \cref{eq:fd_scheme_multi}, where the $\mathcal{F}_k$-exact FSBP operator $D_k$ on the element $I_k$ has design order $p_k$.
	Then the volume terms satisfy $\vol_{k,i} = \partial_x \boldsymbol{f}(\boldsymbol{u}(x_{k,i})) + \mathcal{O}(h^{p})$ with $p = \min_k p_k$.
	That is, the coupled scheme inherits the smallest element-wise design order.
\end{corollary}

\begin{proof}
	Apply \cref{prop:accuracy} to each element, recalling that all estimates are uniform over the elements.
\end{proof}

\begin{remark}[Mixing function spaces does not reduce the order]\label{rem:mixing_spaces}
	By \cref{prop:design_order}, every function space of dimension $p + 1$ that satisfies the Wronskian criterion yields the same design order $p$.
	Assigning different such spaces to different elements---for instance, polynomial spaces where the solution is smooth and slowly varying, and exponential spaces near steep gradients---therefore does not reduce the design order of the coupled scheme, while the error constant can be tuned element by element as described in \cref{rem:tailored_accuracy}.
\end{remark}
 
\section{Computational examples}
\label{sec:numerics}

We now demonstrate the theory of \Cref{sec:fluxDiff,sec:accuracy} numerically.
The experiments are designed to investigate the discrete mass conservation properties (\cref{cor:conservation}), the semi-discrete entropy conservation and stability (\cref{thm:entropy_cons,thm:coupled}), and the design-order accuracy of the volume term (\cref{prop:accuracy}).

\subsubsection*{Implementation and diagnostics}
\label{subsec:num_setup}

We provide some implementation and diagnostic details.
All our computational experiments can be reproduced using the publicly available reproducibility repository \cite{haase2026reproducibility}.
Notably, our codebase \cite{haase2026reproducibility} builds on the Julia package Trixi.jl \cite{ranocha2022adaptive} and extends its functionality from spatially static polynomial SBP to spatially varying FSBP operators.

\subsubsection*{Function spaces and operators}

As a baseline, we use the polynomial space $\mathcal{P}_d = \mathrm{span}\{1, x, \dots, x^d\}$, which recovers the classical polynomial SBP operators.
The FSBP operators are based on the exponential and trigonometric spaces
\begin{align}
	\mathcal{E}_d & = \mathrm{span}\{1, x, \dots, x^{d-1}, e^{x}\}, \label{eq:exp_space} \\ 
	\mathcal{E}^* & = \mathrm{span}\{1, x, e^{x}, e^{2x}\}, \label{eq:exp_space_exact} \\ 
	\mathcal{T}^* & = \mathrm{span}\{1, x, \sin 2\pi x, \cos 2\pi x\}. \label{eq:trig_space_exact}
\end{align}
Each space contains the constant function, as required by our standing assumption.

We construct closed Gaussian FSBP operators \cite{hale2026summation,bercik2026construction,glaubitz2026gaussian} with grids that include the boundary nodes. 
The grid points and the underlying positive and $(\mathcal{F}\mathcal{F})'$-exact quadrature for the norm matrix $P$ correspond to a generalized Gaussian quadrature that we compute using a modified version of the algorithms in \cite{hale2026summation, yarvinGeneralizedGaussianQuadratures1998}.

The domain is partitioned into $K$ elements coupled through numerical interface fluxes.
The approximation space of an element is the global approximation space restricted to the element's domain.
Each element may require a different FSBP operator adapted to its domain---if the function space $\mathcal{F}$ is not translation invariant.
Because the considered function spaces $\mathcal{E}_d, \mathcal{E}^*$ and $\mathcal{T}^*$ are translation invariant, however, it suffices to compute operators for different element lengths.

\subsubsection*{Two-point fluxes}

For Burgers' equation with entropy $S(u) = u^2/2$, we use the entropy-conservative volume flux $\fvol(u_L, u_R) = (u_L^2 + u_L u_R + u_R^2)/6$, which satisfies the EC condition \cref{eq:EC_condition} for the entropy variable $w = u$; see \cite{tadmor1987numerical}.
For the compressible Euler equations, we use the entropy-conservative and kinetic-energy-preserving flux of \cite{ranocha2018comparison} in both the volume and interface terms.
The entropy-stable interface flux is then obtained by adding the matrix dissipation \cref{eq:ES_flux} acting on the jump in the entropy variables, for which we take a modified Roe-type \cite{wintersUniquelyDefinedEntropy2017} (Euler) or local Lax--Friedrichs (Burgers) dissipation operator $\Lambda$.

\subsubsection*{Time integration}
We integrate in time using (i) Verner's explicit ninth-order Runge--Kutta (RK) method \cite{vernerNumericallyOptimalRunge2010} for the convergence studies; 
(ii) an explicit strong-stability-preserving RK (SSP-RK) method \cite{shu1988total} for the robustness studies; and 
(iii) a relaxation RK method \cite{ranocha2020relaxation} for the entropy-conservation experiments, which renders the fully discrete scheme entropy-conservative and thereby allows us to verify \cref{thm:entropy_cons,thm:coupled} to machine precision rather than only up to the time-integration error.
For all integrators, we pick a time step small enough that temporal error does not contaminate the reported spatial quantities.
We select a ninth-order accurate time integration scheme for the convergence studies because one of the operators considered yields an exact semi-discretization. 
Thus, the fully discrete error comes only from the time integrator.

\subsubsection*{Diagnostics}

We measure errors in the discrete $P$-norm and report the relative global error as
\begin{equation}\label{eq:global_P_error}
    e = \frac{ \| \mathbf{u} - \mathbf{u}_{\mathrm{ex}} \|_P }{ \| \mathbf{u}_{\mathrm{ex}} \|_P} = \frac{\sqrt{ \sum_k {\| \mathbf{u}_k - \mathbf{u}_{k,\mathrm{ex}} \|_{P_k}}^2 } }{ \sqrt{ \sum_k {\| \mathbf{u}_{k,\mathrm{ex}} \|_{P_k}}^2}},
\end{equation}
where $\mathbf{u}_{k,\mathrm{ex}}$ are the nodal values of the exact solution at the grid points of element $k$ with norm matrix $P_k$, dropping the index for the global formulation.
Moreover, we report the experimental order of convergence (EOC)
\begin{equation}\label{eq:EOC}
	\mathrm{EOC} = \frac{\log\!\big(e^{(K_1)} / e^{(K_2)} \big)}{\log(K_2 / K_1)},
\end{equation}
where $e^{(K_1)}$ and $e^{(K_2)}$ are the relative global errors \cref{eq:global_P_error} for $K_1$ and $K_2$ elements.

To monitor the structure-preserving properties, we track the discrete mass $\mathbf{1}^T P \mathbf{u} = \sum_k \mathbf{1}^T P_k \mathbf{u}_k$ (\cref{cor:conservation}) and the total entropy $\mathbf{1}^T P\, S(\mathbf{u}) = \sum_k \mathbf{1}^T P_k\, S(\mathbf{u}_k)$ (\cref{thm:entropy_cons,thm:coupled}).
In general, the problem may carry a source $\boldsymbol{s} = \boldsymbol{s}(x,t)$, changing the conservation law \cref{eq:cons_law} into a balance law
\begin{equation}\label{eq:balance_law}
	\partial_t \boldsymbol{u} + \partial_x \boldsymbol{f}(\boldsymbol{u}) = \boldsymbol{s}, \quad x \in (x_L, x_R), \quad t > 0.
\end{equation} 
The source $\boldsymbol{s}$ then enters the semi-discretization as the nodal term $\mathbf{s}$ with rows $\boldsymbol{s}_i = \boldsymbol{s}(x_i,t)$.
Carrying this source through the same contractions as in the proofs of \cref{cor:conservation,thm:entropy_cons} yields the discrete balances
\begin{equation}\label{eq:source_balance}
\begin{aligned}
	\frac{\d}{\d t}\, \mathbf{1}^T P \mathbf{u}
		&= \big( \boldsymbol{f}^{\mathrm{surf}}_L - \boldsymbol{f}^{\mathrm{surf}}_R \big) + \mathbf{1}^T P \mathbf{s}, \\
	\frac{\d}{\d t}\, \mathbf{1}^T P\, S(\mathbf{u})
		&= \big( \psi_N - \boldsymbol{w}_N^T \boldsymbol{f}^{\mathrm{surf}}_R \big) - \big( \psi_1 - \boldsymbol{w}_1^T \boldsymbol{f}^{\mathrm{surf}}_L \big) + \sum_i p_i\, \boldsymbol{w}_i^T \boldsymbol{s}_i ,
\end{aligned}
\end{equation}
where the boundary terms are collected over the physical domain boundary and vanish for periodic boundaries.
For entropy-stable interface fluxes, the entropy balance only holds as an inequality (\cref{thm:coupled}).
Since every term on the right-hand sides of \cref{eq:source_balance} is computable from the boundary states and the source, the primary diagnostic of interest is the time-integrated \emph{residual} of \cref{eq:source_balance}. 
That is, the discrete change of each invariant minus its predicted right-hand side:
\begin{equation}\label{eq:source_balance_integrated}
\begin{aligned}
	\mathbf{1}^T P \mathbf{u}(t_1) - \mathbf{1}^T P \mathbf{u}(t_0)
        &= \int_{t_0}^{t_1} \left[ \big( \boldsymbol{f}^{\mathrm{surf}}_L - \boldsymbol{f}^{\mathrm{surf}}_R \big) + \mathbf{1}^T P \mathbf{s} \right]\intd t, \\
	\mathbf{1}^T P\, S(\mathbf{u}(t_1)) - \mathbf{1}^T P\, S(\mathbf{u}(t_0))
        &= \int_{t_0}^{t_1} \left[ \big( \psi_N - \boldsymbol{w}_N^T \boldsymbol{f}^{\mathrm{surf}}_R \big) - \big( \psi_1 - \boldsymbol{w}_1^T \boldsymbol{f}^{\mathrm{surf}}_L \big) + \sum_i p_i\, \boldsymbol{w}_i^T \boldsymbol{s}_i\right] \intd t .
\end{aligned}
\end{equation}
Denoting the two right-hand sides of \cref{eq:source_balance_integrated} by $\pi_{\mathbf{u}} \in \R^m$ and $\pi_S \in \R$, the mass and entropy residuals are
\begin{equation}\label{eq:residuals}
\begin{aligned}
	r_{\mathbf{u}}(t_1, t_0) 
		& := \mathbf{1}^T P \mathbf{u}(t_1) - \mathbf{1}^T P \mathbf{u}(t_0) - \pi_{\mathbf{u}}, \\
	r_S(t_1, t_0) 
		& := \mathbf{1}^T P\, S(\mathbf{u}(t_1)) - \mathbf{1}^T P\, S(\mathbf{u}(t_0)) - \pi_S - \sigma_S.
\end{aligned}
\end{equation} 
Notably, for the entropy residual $r_S$, we have removed the entropy $\sigma_S = \int_{t_0}^{t_1} \dot S_{\mathrm{exact}} \intd t$ produced by any shocks present in the exact solution, which is nonzero only for the modified Sod tube for which we give a closed form of $\dot S_{\mathrm{exact}}$ in \Cref{subsec:num_euler_sod}.
By \cref{cor:conservation,thm:entropy_cons}, $r_{\mathbf{u}}$ vanishes for any consistent and symmetric flux, and for $\sigma_S = 0$ so does $r_S$ for an entropy-conservative flux. 
So a nonzero residual reflects time integration, round-off, and, for entropy-stable fluxes, the interface dissipation.

For a source-free problem ($\boldsymbol{s} = 0$), on a periodic domain, the right-hand sides of \cref{eq:source_balance_integrated} vanish. 
The discrete mass is then exactly constant, and the total entropy is constant for an entropy-conservative flux or non-increasing for an entropy-stable flux.
In this case, the residuals reduce to the differences $\mathbf{1}^T P \mathbf{u}(t_1) - \mathbf{1}^T P \mathbf{u}(t_0)$ and $\mathbf{1}^T P\, S(\mathbf{u}(t_1)) - \mathbf{1}^T P\, S(\mathbf{u}(t_0))$.
We use this homogeneous setting for the dedicated conservation and entropy tests, and the boundary-corrected residuals \cref{eq:residuals} for the modified Sod shock tube of \Cref{subsec:num_euler_sod}.

\subsection{Burgers' equation}
\label{subsec:num_burgers}

We consider the inviscid Burgers' equation $\partial_t u + \partial_x (u^2/2) = 0$ with the square entropy $S(u) = u^2/2$, which is the simplest setting exhibiting genuinely nonlinear entropy conservation and stability. 
We take the periodic domain $[-1, 1]$ with initial data
\begin{equation}\label{eq:burgers_ic}
	u(x, 0) = c - \sin\big( \pi (x - x_s) \big), 
	\quad x_s = \tfrac{1}{8}, \quad c = \tfrac{3}{2},
\end{equation}
so that a shock forms at $t^\ast = 1/\pi$ and subsequently travels at speed $c$.

\subsubsection*{Smooth regime}

Before shock formation, the solution is smooth, and we report the corresponding EOCs in \cref{tab:burgers_eoc,fig:burgers_conv}, comparing $\mathcal{P}_d$, $\mathcal{E}^*$, and $\mathcal{E}_d$ with entropy-stable surface fluxes.

\begin{figure}[tb]
	\centering
	\begin{subfigure}{0.49\linewidth}
		\centering
		\includegraphics{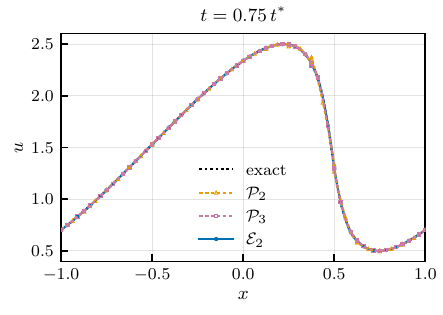}
		\caption{Solutions at $t = 0.75 t^\ast$.}
		\label{fig:burgers_sol_smooth}
	\end{subfigure}
	\begin{subfigure}{0.49\linewidth}
		\centering
		\includegraphics{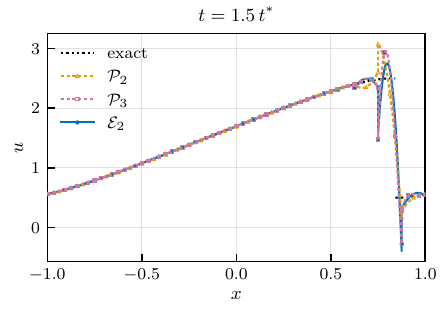}
		\caption{Solutions at $t = 1.5 t^\ast$.}
		\label{fig:burgers_sol_shock}
	\end{subfigure}
    \caption{Comparison of numerical solutions to Burgers' equation with initial data \cref{eq:burgers_ic} against the exact solution before and after the time of shock formation $t^\ast$. Shown are solutions obtained with the FSBP operator for the space $\mathcal{E}_2 = \mathrm{span}\{1, x, e^x\}$ and the standard DGSEM polynomial operators for $\mathcal{P}_2$ and $\mathcal{P}_3$, all on $16$ elements and with Verner's RK method in time. The markers indicate the nodal solution of the schemes. For $\mathcal{P}_2$ and $\mathcal{P}_3$ the dashed lines show the corresponding polynomials, while for $\mathcal{E}_2$ we show the least-squares fit to the nodal solution.}
	\label{fig:burgers_snapshots}
\end{figure}

\begin{table}[tb]
	\centering
	\small
	\setlength{\tabcolsep}{4.5pt}
    \caption{Relative errors and EOCs for Burgers' equation before shock formation at $t = 0.75 t^\ast$. 
    We compare operators with the same number of nodes. 
    The FSBP operator has $4$ nodes for $\mathcal{E}_2 = \mathrm{span}\{1, x, e^x\}$ and $5$ nodes for $\mathcal{E}^{*} = \mathrm{span}\{1, x, e^x, e^{2x}\}$. 
    We compare them with SBP operators for the polynomial spaces $\mathcal{P}_3$ and $\mathcal{P}_4$ with $4$ and $5$ nodes, respectively. 
    In bold, we show the lower relative error at the same number of nodes.}
	\label{tab:burgers_eoc}
	\begin{tabular}{c cc cc cc cc c c}
	\toprule
	 & \multicolumn{4}{c}{$4$ nodes} & \multicolumn{4}{c}{$5$ nodes} &  &  \\
	\cmidrule(lr){2-5}\cmidrule(lr){6-9}
	 & \multicolumn{2}{c}{$\mathcal{P}_3$ ($\dim \mathcal{P}_3 = 4$)} & \multicolumn{2}{c}{$\mathcal{E}_2$ ($\dim \mathcal{E}_2 = 3$)} & \multicolumn{2}{c}{$\mathcal{P}_4$ ($\dim \mathcal{P}_4 = 5$)} & \multicolumn{2}{c}{$\mathcal{E}^{*}$ ($\dim \mathcal{E}^{*} = 4$)} & $\mathcal{P}_3/\mathcal{E}_2$ & $\mathcal{P}_4/\mathcal{E}^{*}$ \\
	\cmidrule(lr){2-3}\cmidrule(lr){4-5}\cmidrule(lr){6-7}\cmidrule(lr){8-9}
	$K$ & error & EOC & error & EOC & error & EOC & error & EOC &  &  \\
	\midrule
	4 & $\mathbf{2.52\times10^{-2}}$ & --- & $2.57\times10^{-2}$ & --- & $1.42\times10^{-2}$ & --- & $\mathbf{1.33\times10^{-2}}$ & --- & $0.98$ & $1.07$ \\
	8 & $9.68\times10^{-3}$ & $1.38$ & $\mathbf{9.55\times10^{-3}}$ & $1.43$ & $2.99\times10^{-3}$ & $2.25$ & $\mathbf{2.62\times10^{-3}}$ & $2.34$ & $1.01$ & $1.14$ \\
	16 & $1.70\times10^{-3}$ & $2.51$ & $\mathbf{1.65\times10^{-3}}$ & $2.53$ & $\mathbf{2.00\times10^{-4}}$ & $3.90$ & $2.31\times10^{-4}$ & $3.51$ & $1.03$ & $0.87$ \\
	32 & $\mathbf{7.94\times10^{-5}}$ & $4.42$ & $8.28\times10^{-5}$ & $4.32$ & $\mathbf{6.95\times10^{-5}}$ & $1.53$ & $6.97\times10^{-5}$ & $1.73$ & $0.96$ & $1.00$ \\
	64 & $\mathbf{2.50\times10^{-5}}$ & $1.67$ & $2.52\times10^{-5}$ & $1.72$ & $2.46\times10^{-6}$ & $4.82$ & $\mathbf{2.38\times10^{-6}}$ & $4.87$ & $0.99$ & $1.04$ \\
	128 & $\mathbf{1.77\times10^{-6}}$ & $3.82$ & $1.78\times10^{-6}$ & $3.83$ & $\mathbf{9.64\times10^{-8}}$ & $4.68$ & $9.68\times10^{-8}$ & $4.62$ & $1.00$ & $1.00$ \\
	256 & $\mathbf{1.18\times10^{-7}}$ & $3.90$ & $1.19\times10^{-7}$ & $3.90$ & $\mathbf{3.23\times10^{-9}}$ & $4.90$ & $3.24\times10^{-9}$ & $4.90$ & $1.00$ & $1.00$ \\
	512 & $\mathbf{7.61\times10^{-9}}$ & $3.96$ & $7.63\times10^{-9}$ & $3.96$ & $\mathbf{1.04\times10^{-10}}$ & $4.96$ & $1.04\times10^{-10}$ & $4.96$ & $1.00$ & $1.00$ \\
	\bottomrule
\end{tabular}

\end{table}

\begin{figure}[tb]
	\centering
	\includegraphics{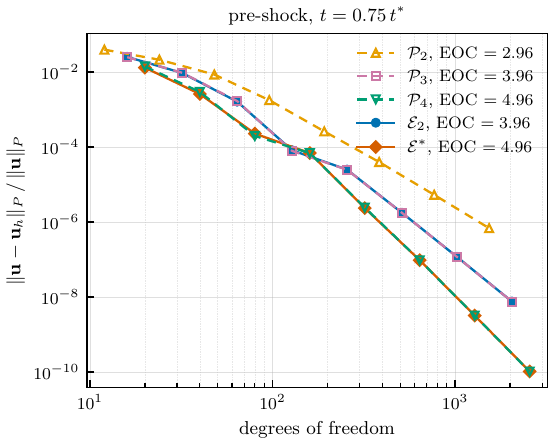}
    \caption{Convergence study for solutions obtained with FSBP operators for the spaces $\mathcal{E}_2 = \mathrm{span}\{1, x, e^x\}$ and $\mathcal{E}^{*} = \mathrm{span}\{1, x, e^x, e^{2x}\}$ against standard DGSEM SBP operators for the polynomial spaces $\mathcal{P}_2$, $\mathcal{P}_3$ and $\mathcal{P}_4$ for Burgers' equation before shock formation at $t = 0.75 t^\ast$ with initial data \cref{eq:burgers_ic}. 
    We compute the EOCs from the last two data points.}
	\label{fig:burgers_conv}
\end{figure}

We stop at $t = 0.75 t^\ast$, which yields the numerical solutions in \cref{fig:burgers_sol_smooth}.
In \cref{tab:burgers_eoc}, we observe that the polynomial operators for the spaces $\mathcal{P}_d$ achieve $\text{EOC} \approx d + 1$ asymptotically, while $\mathcal{E}_2$ reaches $\text{EOC} \approx 4$ and $\mathcal{E}^*$ reaches $\text{EOC} \approx 5$.

\begin{remark}[Higher order than guaranteed]\label{rem:eoc_vs_design_order}
	The EOCs observed for the FSBP operators are higher than the design order guaranteed by \cref{prop:design_order} suggests.
	For the polynomial operators, the design order $p = d = N - 1$ (\cref{expl:design_order_poly}) comes with $\text{EOC} \approx p + 1 = N$, the familiar behavior of polynomial SBP operators on $N$ Gauss--Lobatto nodes.
	For the FSBP operators, \cref{prop:design_order} only guarantees the design order $\dim \mathcal{F} - 1$, i.e., $p = 2$ for $\mathcal{E}_2$ and $p = 3$ for $\mathcal{E}^*$, and by analogy with the polynomial case one would expect $\text{EOC} \approx 3$ and $\text{EOC} \approx 4$.
	Instead, we observe $\text{EOC} \approx 4$ for $\mathcal{E}_2$ on $N = 4$ nodes and $\text{EOC} \approx 5$ for $\mathcal{E}^*$ on $N = 5$ nodes.
	Our experiments thus indicate that Gaussian FSBP operators can have design order $N - 1$, determined by the number of nodes rather than by $\dim \mathcal{F} - 1$, just as for the polynomial DGSEM operators, for which $N - 1 = \dim \mathcal{P}_{N-1} - 1$ so that both agree.
	Since our construction requires $N = \dim \mathcal{F} + 1$ nodes for $\mathcal{E}_2$ and $\mathcal{E}^*$, this is one order more than the currently existing theory anticipates.
	We hope to establish a deeper theoretical understanding of this finding in future work, closing the gap between the theory for FSBP operators and their surprisingly higher convergence orders in practice.
	Such a theory would sharpen rather than change the conclusion of \cref{rem:tailored_accuracy}: asymptotically, FSBP operators are neither worse nor better than polynomial SBP operators on the same number of nodes.
	What we aim at with tailored function spaces is a smaller error constant, i.e., improved accuracy for coarse-grid simulations.
\end{remark}

On a periodic domain, the discrete mass $\mathbf{1}^T P \mathbf{u}$ is conserved to high precision at errors below $10^{-13}$ (\cref{cor:conservation}).
In \cref{fig:burgers_mass}, we observe a small upward drift of the conservation error for the operators for $\mathcal{P}_4$ and $\mathcal{E}^*$, which are the only ones with errors at about $10^{-14}$, while the errors for $\mathcal{P}_2$, $\mathcal{P}_3$ and $\mathcal{E}_2$ remain at machine precision.
We elaborate on this phenomenon in \Cref{subsec:num_euler_density_wave}.

\begin{figure}[tb]
	\centering
	\includegraphics{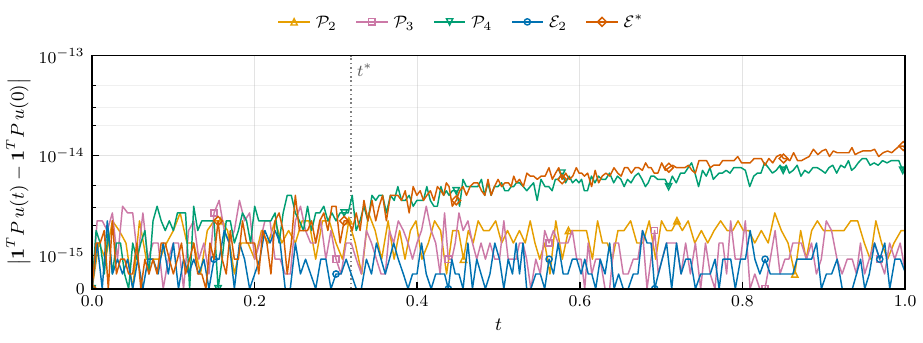}
    \caption{Conservation errors of numerical solutions to Burgers' equation with initial data \cref{eq:burgers_ic}. We compare solutions obtained with FSBP operators for the spaces $\mathcal{E}_2 = \mathrm{span}\{1, x, e^x\}$ and $\mathcal{E}^{*} = \mathrm{span}\{1, x, e^x, e^{2x}\}$ against solutions obtained with standard DGSEM polynomial SBP operators $\mathcal{P}_2$, $\mathcal{P}_3$ and $\mathcal{P}_4$. 
    We used $16$ elements and the relaxation RK method for time stepping and entropy-stable surface fluxes.}
	\label{fig:burgers_mass}
\end{figure}

Using the entropy-conservative flux together with relaxation time integration, \cref{fig:burgers_entropy} demonstrates that the total entropy is conserved to up to round-off over time (\cref{thm:coupled}), whereas the entropy-stable flux produces a monotonically non-increasing total entropy.
While physically unsound, the entropy-conservative scheme does not crash after the shock formation at $t^\ast$ and preserves entropy.

\begin{figure}[tb]
	\centering
	\includegraphics{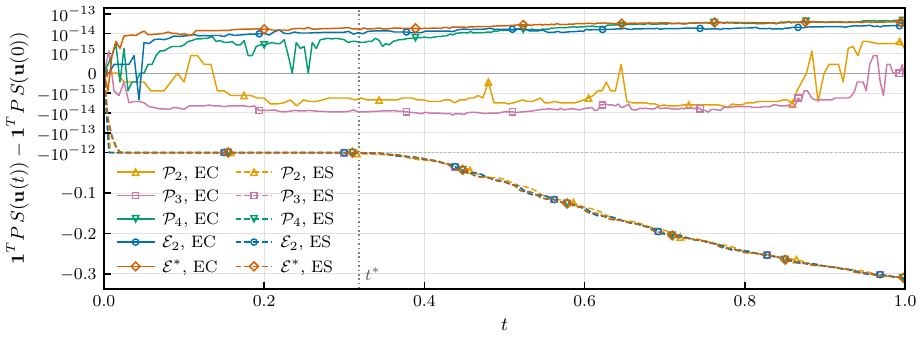}
	\caption{Comparison of entropy dissipation of numerical solutions to Burgers' equation with initial data \cref{eq:burgers_ic}. 
	We compare solutions obtained with FSBP operators for the spaces $\mathcal{E}_2 = \mathrm{span}\{1, x, e^x\}$ and $\mathcal{E}^{*} = \mathrm{span}\{1, x, e^x, e^{2x}\}$ against solutions obtained with standard DGSEM polynomial SBP operators $\mathcal{P}_2$, $\mathcal{P}_3$ and $\mathcal{P}_4$. 
	Furthermore, we compare the behavior of the numerical solutions for entropy-conservative (EC) and entropy-stable surface fluxes (ES). 
	In the upper half of the plot, we use a symmetric logarithmic scale to show the accuracy of entropy-conservative surface fluxes in detail. 
	In the lower half, we use a linear scale showing how entropy-stable surface fluxes dissipate entropy. 
	We used $16$ elements and the relaxation RK method for time stepping to achieve machine-precision entropy conservation for entropy-conservative surface fluxes.}
	\label{fig:burgers_entropy}
\end{figure}

\subsubsection*{Shock regime}

After shock formation (as in \cref{fig:burgers_sol_shock}), we compare entropy-stable schemes against the exact solution obtained from the Lax formula \cite[Sec.~11.4]{dafermos2016hyperbolic}.
At time $t = 1.5 t^\ast$ we find that the numerical solutions for $\mathcal{P}_3$ and $\mathcal{E}_2$ are almost indistinguishable, both tracking the shock well.
The operator for $\mathcal{P}_2$ behaves similarly, displaying slightly different oscillations around the shock.
We do not show the entropy-conservative numerical solutions because they are nearly dissipation-free, yielding highly oscillatory approximations that do not match the exact solution.

\subsection{Manufactured solution for Burgers' equation}
\label{subsec:num_burgers_mms}

The previous experiment confirms the structure-preserving properties. 
We now isolate the effect of the approximation space on accuracy.
For this, we use the method of manufactured solutions, prescribing a family of smooth exact solutions with an exponential spatial profile and adding the matching source.
On $\Omega = [0,1]$ and $t \in [0,1]$, we take
\begin{equation}\label{eq:burgers_mms}
	u_{\mathrm{ex}}(x,t) = a(t)\, e^{\alpha x} + b(t),
\end{equation}
and consider the inhomogeneous Burgers' equation $\partial_t u + \partial_x (u^2/2) = s$ with source term
\begin{equation}\label{eq:burgers_mms_source}
	s(x,t) = \partial_t u_{\mathrm{ex}} + \partial_x \big( u_{\mathrm{ex}}^2 / 2 \big).
\end{equation}
Furthermore, we take the boundary data from $u_{\mathrm{ex}}$ for the inflow on the left and impose it weakly through the surface flux.
At the same time, we use the entropy-stable interface flux throughout. 
We remark that, when running the subsequent accuracy study for the entropy-conservative interface flux, we observe the same odd-even behavior of the convergence order as described in \cite{hindenlangOrderReductionEntropy2020} for polynomial SBP operators and the FSBP operators for $\mathcal{E}_2$.
In our numerical examples, we use $a(t) = 1 - e^{-5t}$ and $b(t) = e^{-5t}$, starting from the constant state $u_{\mathrm{ex}}(\cdot, 0) = 1$ and relaxing toward the steady state $u_\infty = e^{\alpha x}$ as $t$ increases.
We consider the rates $\alpha \in \{\tfrac{1}{2}, 1, 2\}$, which set the steepness of the exponential profile.
In \cref{fig:burgers_mms_profile} the relaxation to $e^{\alpha x}$ is shown for $\alpha = 1$.

\begin{figure}[tb]
	\centering
	\includegraphics{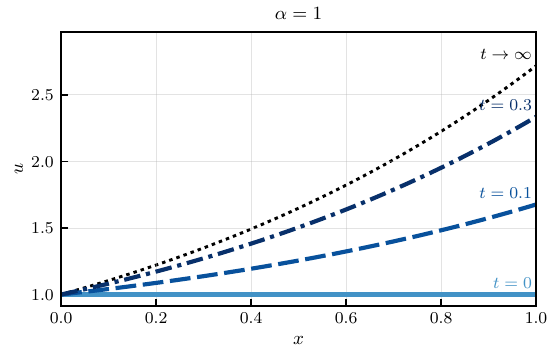}
    \caption{Manufactured solution $u_\mathrm{ex} = (1-e^{-5 t})e^x + e^{-5t}$ described in \cref{eq:burgers_mms} at different times for $\alpha = 1$.}
	\label{fig:burgers_mms_profile}
\end{figure}

The space $\mathcal{E}_2$ is tailored to the problem at $\alpha = 1$ because then the solution is in the span of $\{1, e^x\}$.
Our construction procedure requires $N = 4$ nodes per element for the FSBP operator for $\mathcal{E}_2$---one more node than $\dim \mathcal{E}_2 = 3$.
For this reason, we compare the FSBP operator against the polynomial DGSEM operators for $\mathcal{P}_2$ and $\mathcal{P}_3$ with $N = 3$ and $N = 4$ nodes per element in \cref{fig:burgers_mms}.
\Cref{tab:burgers_mms_eoc} lists the corresponding errors and EOCs for the operators for $\mathcal{P}_3$ and $\mathcal{E}_2$ with $4$ nodes and for $\mathcal{P}_4$ and $\mathcal{E}^{*}$ with $5$ nodes.
At the same number of degrees of freedom, the FSBP operator for $\mathcal{E}_2$ is uniformly more accurate than the operator for $\mathcal{P}_3$ (by a factor of $\approx 2.29$) and attains fourth-order convergence---matching the cubic operator $\mathcal{P}_3$ of equal cost---whereas the quadratic operator $\mathcal{P}_2$ is only third order.
The gap between the almost parallel lines of $\mathcal{E}_2$ and $\mathcal{P}_3$ in \cref{fig:burgers_mms} for $\alpha=1$ visualizes the improved factor.
For $\alpha = 1/2$, this gap is even larger, indicating that $\mathcal{E}_2$ performs even better than at $\alpha = 1$, despite a mismatch between solution space and approximation space.
A possible explanation is that for $\alpha = 1/2$ the flux derivative of the steady state is in $\mathcal{E}_2$.
For $\alpha = 2$, the accuracy gain of $\mathcal{E}_2$ shrinks substantially, and the EOC line plots almost overlap.
This motivates the construction of an operator for $\mathcal{E}^{*} = \mathrm{span}\{1, x, e^x, e^{2x}\}$, which provides exact representations of the flux derivative and the solution for $\alpha = 1$.
\Cref{fig:burgers_mms,tab:burgers_mms_eoc} confirm our choice: we obtain a numerical solution at machine precision, hence outperforming any other operator by a large margin.

\begin{table}[tb]
	\centering
	\small
	\setlength{\tabcolsep}{4.5pt}
    \caption{Relative errors and EOC for numerical solutions to Burgers' equation with a manufactured solution as in \cref{eq:burgers_mms} with $\alpha = 1$ at $t=1$. 
    We compare operators with the same number of degrees of freedom. 
    The FSBP operators for $\mathcal{E}_2 = \mathrm{span}\{1, x, e^x\}$ has $4$ nodes and the one for $\mathcal{E}^{*} = \mathrm{span}\{1, x, e^x, e^{2x}\}$ has $5$ nodes. 
    We compare them with the operators for the polynomial spaces $\mathcal{P}_3$ and $\mathcal{P}_4$, respectively. 
    In bold, we show the lower relative error at the same number of nodes. 
    We do not show the EOC of $\mathcal{E}^{*}$ and the error ratio for $\mathcal{P}_4$ versus $\mathcal{E}^{*}$ because $\mathcal{E}^{*}$ is (almost) at machine precision from the start.}
	\label{tab:burgers_mms_eoc}
	\begin{tabular}{c ll ll ll l l}
	\toprule
	 & \multicolumn{4}{c}{$4$ nodes} & \multicolumn{3}{c}{$5$ nodes} &  \\
	\cmidrule(lr){2-5}\cmidrule(lr){6-8}
	 & \multicolumn{2}{c}{$\mathcal{P}_3$ ($\dim \mathcal{P}_3 = 4$)} & \multicolumn{2}{c}{$\mathcal{E}_2$ ($\dim \mathcal{E}_2 = 3$)} & \multicolumn{2}{c}{$\mathcal{P}_4$ ($\dim \mathcal{P}_4 = 5$)} & $\mathcal{E}^{*}$ ($\dim \mathcal{E}^{*} = 4$) & $\mathcal{P}_3/\mathcal{E}_2$ \\
	\cmidrule(lr){2-3}\cmidrule(lr){4-5}\cmidrule(lr){6-7}
	$K$ & error & EOC & error & EOC & error & EOC & error &  \\
	\midrule
	2 & $3.94\times10^{-4}$ & --- & $\mathbf{1.72\times10^{-4}}$ & --- & $1.74\times10^{-5}$ & --- & $\mathbf{7.41\times10^{-15}}$ & $2.28$ \\
	4 & $2.35\times10^{-5}$ & $4.06$ & $\mathbf{1.03\times10^{-5}}$ & $4.07$ & $5.23\times10^{-7}$ & $5.06$ & $\mathbf{1.73\times10^{-16}}$ & $2.29$ \\
	8 & $1.43\times10^{-6}$ & $4.04$ & $\mathbf{6.26\times10^{-7}}$ & $4.04$ & $1.60\times10^{-8}$ & $5.03$ & $\mathbf{9.09\times10^{-16}}$ & $2.29$ \\
	16 & $8.84\times10^{-8}$ & $4.02$ & $\mathbf{3.86\times10^{-8}}$ & $4.02$ & $4.92\times10^{-10}$ & $5.02$ & $\mathbf{1.30\times10^{-15}}$ & $2.29$ \\
	32 & $5.48\times10^{-9}$ & $4.01$ & $\mathbf{2.40\times10^{-9}}$ & $4.01$ & $1.53\times10^{-11}$ & $5.01$ & $\mathbf{1.33\times10^{-15}}$ & $2.29$ \\
	64 & $3.42\times10^{-10}$ & $4.01$ & $\mathbf{1.49\times10^{-10}}$ & $4.01$ & $4.76\times10^{-13}$ & $5.00$ & $\mathbf{1.03\times10^{-15}}$ & $2.29$ \\
	128 & $2.13\times10^{-11}$ & $4.00$ & $\mathbf{9.31\times10^{-12}}$ & $4.00$ & $1.68\times10^{-14}$ & $4.82$ & $\mathbf{1.65\times10^{-15}}$ & $2.29$ \\
	\bottomrule
\end{tabular}

\end{table}

\begin{figure}[tb]
	\centering
	\includegraphics{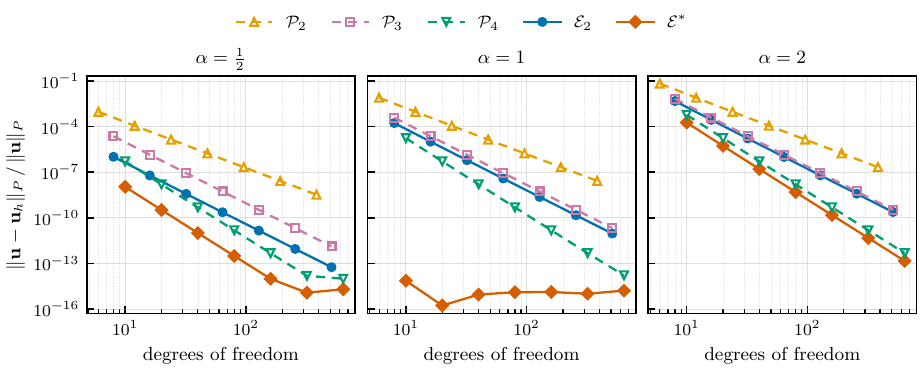}
    \caption{Comparison of the convergence behavior of solutions obtained with FSBP operators for the spaces $\mathcal{E}_2 = \mathrm{span}\{1, x, e^x\}$ and $\mathcal{E}^{*} = \mathrm{span}\{1, x, e^x, e^{2x}\}$ against standard DGSEM polynomial SBP operators for $\mathcal{P}_2$, $\mathcal{P}_3$ and $\mathcal{P}_4$ for Burgers' equation with a manufactured solutions $u_\mathrm{ex} = (1-e^{-5 t})e^{\alpha x} + e^{-5t}$ at $t = 1$ as in \cref{eq:burgers_mms}, for the rates $\alpha \in \{\tfrac{1}{2}, 1, 2\}$ (left to right).}
	\label{fig:burgers_mms}
\end{figure}

\subsection{Compressible Euler equations}
\label{subsec:num_euler1d}

Consider the one-dimensional compressible Euler equations
\begin{equation}\label{eq:euler_1d}
	\partial_t
	\begin{pmatrix} \rho \\ \rho v \\ E \end{pmatrix}
	+ \partial_x
	\begin{pmatrix} \rho v \\ \rho v^2 + p \\ (E + p)\, v \end{pmatrix}
	= 0,
	\quad
	E = \frac{p}{\gamma - 1} + \frac{1}{2} \rho v^2,
\end{equation}
with ratio of specific heats $\gamma = 7/5$ and the mathematical entropy $S = -\rho s / (\gamma - 1)$, where $s = \ln\big( p \rho^{-\gamma} \big)$ is the thermodynamic entropy.

\subsection{Density wave for the compressible Euler equations}
\label{subsec:num_euler_density_wave}

We take the periodic domain $\Omega = [-1, 1]$ with initial data as in \cite[Sec.~3]{gassnerStabilityIssuesEntropyStable2022}:
\begin{equation}\label{eq:euler_density_wave}
	\rho(x, 0) = 1 + 0.5 \sin(2 \pi x), \quad
	v(x, 0) = \tfrac{1}{10}, \quad
	p(x, 0) = 20,
\end{equation}
where we reduce the amplitude from $0.98$ to $0.5$ to keep the density bounded away from vacuum.
Velocity and pressure remain constant, so that \cref{eq:euler_1d} reduces to linear advection of the density and
\begin{equation}\label{eq:euler_density_wave_exact}
	\rho_{\mathrm{ex}}(x, t) = 1 + 0.5 \sin\big( 2 \pi (x - v t) \big)
\end{equation}
is the exact solution for all $t$.

\subsubsection*{Conservation and entropy}

For a smooth, periodic flow integrated with the entropy-conservative flux and relaxation time stepping, the polynomial (SBP), trigonometric (FSBP) and exponential (FSBP) operators conserve the discrete mass $\mathbf{1}^T P \bm{\rho}$ up to a small error (\cref{cor:conservation}), as illustrated in \cref{fig:euler_mass}.
The mass conservation errors of $\mathcal{P}_4$ and $\mathcal{E}^{*}$ increase over time, reaching up to $10^{-13}$ at $t=100$.
We suspect that this increase is due to floating-point inaccuracies and the accumulation of rounding errors.
To strengthen this hypothesis, we reran the same simulation at higher accuracy using more bits for floating-point representation and observed that the errors decreased for all conserved variables.

\begin{figure}[tb]
  \centering
  \includegraphics{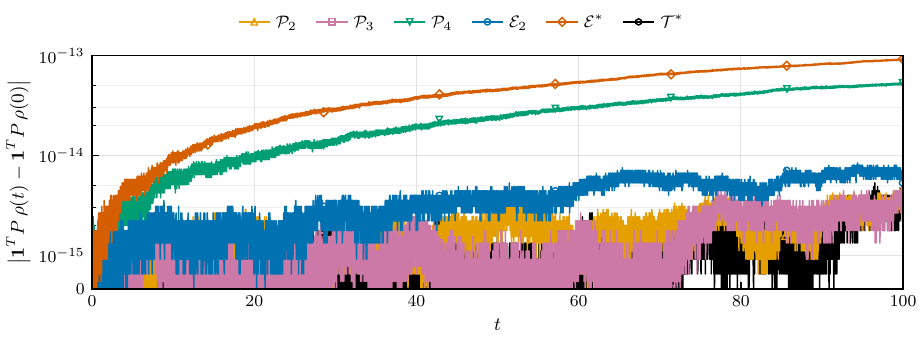}
    \caption{Errors in the discrete mass $\mathbf{1}^T P \boldsymbol{\rho}$ of the numerical solutions to the compressible Euler equations for the density wave \cref{eq:euler_density_wave}. 
    We compare solutions obtained with FSBP operators for the spaces $\mathcal{E}_2 = \mathrm{span}\{1, x, e^x\}$, $\mathcal{E}^{*} = \mathrm{span}\{1, x, e^x, e^{2x}\}$ and $\mathcal{T}^{*} = \mathrm{span}\{1, x, \sin 2\pi x, \cos 2\pi x\}$ against solutions obtained with standard DGSEM polynomial SBP operators $\mathcal{P}_2$, $\mathcal{P}_3$ and $\mathcal{P}_4$. 
    We used $16$ elements and the relaxation RK method for time stepping and entropy-conservative surface fluxes.}
  \label{fig:euler_mass}
\end{figure}

\cref{fig:euler_entropy} shows the corresponding changes in the total entropy $\mathbf{1}^T P\, S(\mathbf{u})$.
We observe that the entropy errors are bounded at magnitude $\approx 10^{-12}$, confirming \cref{thm:coupled}.

\begin{figure}[tb]
  \centering
  \includegraphics{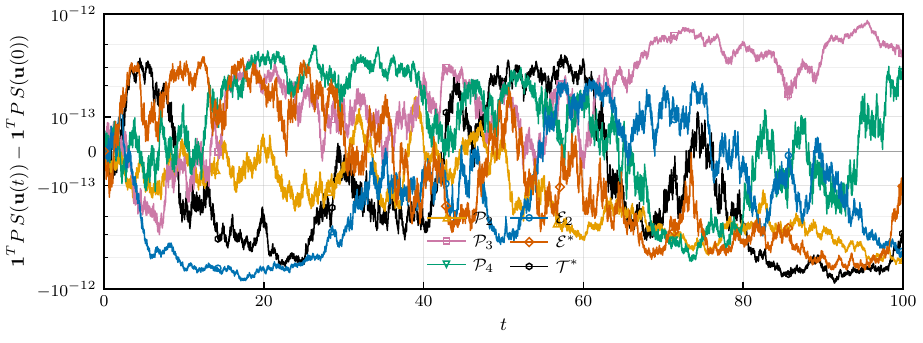}
    \caption{Changes in the total entropy $\mathbf{1}^T P\, S(\mathbf{u})$ of numerical solutions to the compressible Euler equations for the density wave \cref{eq:euler_density_wave}. 
    We compare solutions obtained with FSBP operators for the spaces $\mathcal{E}_2 = \mathrm{span}\{1, x, e^x\}$, $\mathcal{E}^{*} = \mathrm{span}\{1, x, e^x, e^{2x}\}$ and $\mathcal{T}^{*} = \mathrm{span}\{1, x, \sin 2\pi x, \cos 2\pi x\}$ against solutions obtained with standard DGSEM polynomial SBP operators $\mathcal{P}_2$, $\mathcal{P}_3$ and $\mathcal{P}_4$. 
    The vertical axis uses a symmetric logarithmic scale to resolve both positive and negative changes. 
    We used $16$ elements and the relaxation RK method for time stepping and entropy-conservative surface fluxes.}
  \label{fig:euler_entropy}
\end{figure}

\subsubsection*{Accuracy}

The convergence plot in \cref{fig:euler_accuracy} shows that the polynomial SBP operators of degree $d$ have an asymptotic EOC of $d+1$, matching their FSBP counterparts at the same node count---recall \cref{rem:eoc_vs_design_order}.
While $\mathcal{P}_3$ and $\mathcal{E}_2$ with $4$ nodes and $\mathcal{P}_4$ and $\mathcal{E}^{*}$ with $5$ nodes show almost identical convergence behavior, $\mathcal{T}^{*}$ performs slightly better with a smaller error at the same EOC.
Even though $\mathcal{T}^{*}$ contains the exact solution and the flux derivative at all times, we do not observe the same behavior as for the manufactured solution for Burgers' equation with $\mathcal{E}^{*}$.
We attribute this discrepancy to the usage of logarithmic means in the entropy-conservative flux \cite{ranocha2018comparison}, which moves the derivative of the numerical flux out of $\mathcal{T}^{*}$.
Using the standard weak-form DG formulation rather than entropy-conservative flux differencing, we directly differentiate the flux and again retain a semi-discretization accurate up to machine precision.
Details are available in our reproducibility repository \cite{haase2026reproducibility}.

\begin{figure}[tb]
  \centering
  \includegraphics{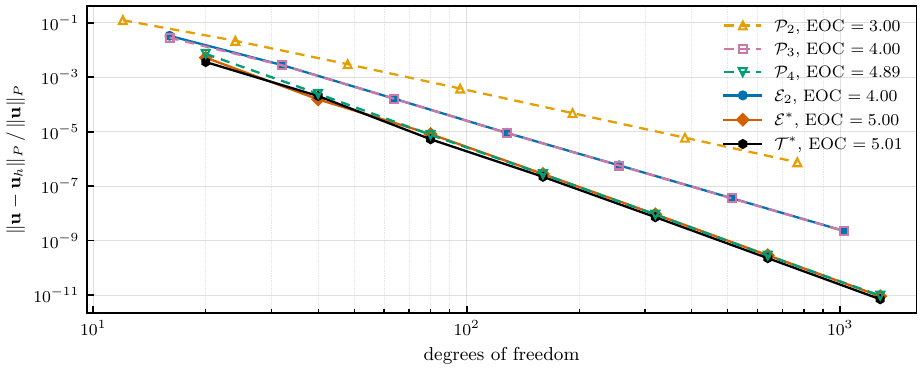}
    \caption{Comparison of the convergence behavior of solutions obtained with our FSBP operators for the spaces $\mathcal{E}_2 = \mathrm{span}\{1, x, e^x\}$, $\mathcal{E}^{*} = \mathrm{span}\{1, x, e^x, e^{2x}\}$ and $\mathcal{T}^{*} = \mathrm{span}\{1, x, \sin 2\pi x, \cos 2\pi x\}$ against standard DGSEM polynomial operators $\mathcal{P}_2$, $\mathcal{P}_3$ and $\mathcal{P}_4$ for the compressible Euler equations with the density wave \cref{eq:euler_density_wave} at $t = 2$. We used entropy-stable surface fluxes. The EOCs are computed from the last two data points. For the Euler equations, we use $\rho$ to compute the errors.}
  \label{fig:euler_accuracy}
\end{figure}

\subsection{Modified Sod shock tube}
\label{subsec:num_euler_sod}

We now consider the one-dimensional compressible Euler equations \cref{eq:euler_1d} on the domain $\Omega = [0, 1]$ and with the Riemann initial data of \cite[Sec.~6.4, Table~6.2, Test~1 and Fig.~6.8]{toro2013riemann}:
\begin{equation}\label{eq:modified_sod}
	(\rho, v, p)(x, 0) =
	\begin{cases}
		(1, \tfrac{3}{4}, 1), & x < x_0, \\
		(\tfrac{1}{8}, 0, \tfrac{1}{10}), & x > x_0,
	\end{cases}
	\quad x_0 = \tfrac{3}{10},
\end{equation}
which we integrate to $t = 1/5$.
Unlike the classical Sod problem, whose left rarefaction is entirely subsonic, the left state moves at $v_L = 3/4$, which makes the left rarefaction transonic. 
The characteristic speed $v - c$ changes sign inside the fan, so that the fan contains the sonic point $v = c$, located at $x = x_0$ for all $t$.
Schemes that violate the entropy condition---classically, linearized Riemann solvers without an entropy fix---can produce a spurious entropy glitch or even an expansion shock at this sonic point \cite[Sec.~11.4]{toro2013riemann}.
No wave reaches a boundary before $t = 1/5$, so we use constant boundary fluxes matching the initial data.
This is the only test with a nonzero $\sigma_S$ in \cref{eq:residuals}.
As the thermodynamic entropy $s$ is constant along particle paths in smooth flows \cite[Prop.~3.8]{toro2013riemann}, the data \cref{eq:modified_sod} with $p_L = \rho_L = 1$ give $S = 0$ left of the contact, so only the two regions right of it contribute.
By self-similarity, the region between the contact and the shock front grows linearly, while the region right of the shock shrinks linearly. 
Hence, $\sigma_S = \dot S_{\mathrm{exact}} (t_1 - t_0)$ with the constant rate
\begin{equation}\label{eq:sod_entropy_rate}
	\dot S_{\mathrm{exact}} = S(\boldsymbol{u}^*_R) \, (\lambda_R - v^*) - S(\boldsymbol{u}_R) \, \lambda_R ,
\end{equation}
where the contact speed $v^*$, the post-shock state $\boldsymbol{u}^*_R$ and the shock speed $\lambda_R$ are taken from the exact Riemann solution \cite[Eqs.~(4.9),~(4.57)~and~(4.59)]{toro2013riemann}.

\subsubsection*{Robustness}

In \cref{fig:euler_sod}, we solve the Sod shock-tube problem with the entropy-stable scheme and compare the three state components (density $\rho$, momentum $\rho v$, and total energy $E$) against the exact solution of the Riemann problem.
The top row shows the polynomial SBP operator and the bottom row the exponential FSBP operator. 
Both capture the shock and contact discontinuity at the correct location but with oscillations.

\begin{figure}[tb]
  \centering
  \includegraphics{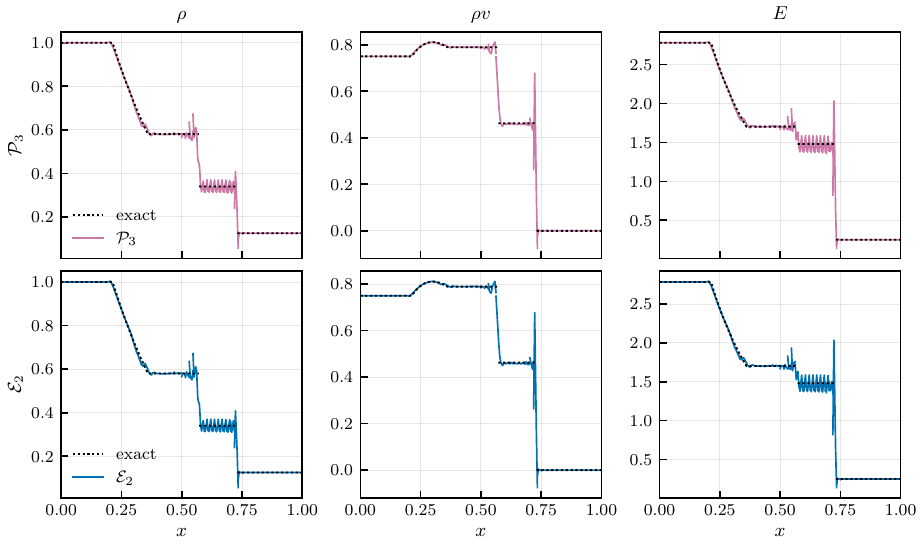}
    \caption{Comparison of numerical solutions to the modified Sod shock tube with initial data \cref{eq:modified_sod} against the exact solution of the Riemann problem at $t = 1/5$. Shown are the density $\rho$, the momentum $\rho v$, and the total energy $E$ obtained with the standard DGSEM polynomial SBP operator for $\mathcal{P}_3$ (top row) and with the FSBP operator for the space $\mathcal{E}_2 = \mathrm{span}\{1, x, e^x\}$ (bottom row), both on $64$ elements with entropy-stable surface fluxes and the SSP-RK method for time stepping.}
  \label{fig:euler_sod}
\end{figure}

\subsubsection*{Conservation}

As for the previous examples, we observe in \cref{fig:euler_sod_conservation} a small increase in the conservation error for both operators.
The magnitude of the error of our exponential FSBP operator for $\mathcal{E}_2$ is comparable to the error of the polynomial operator for $\mathcal{P}_3$.
As expected, the scheme is entropy dissipative, and in the bottom left panel of \cref{fig:euler_sod_conservation}, the entropy is dissipated at an almost constant rate for our numerical schemes, closely matching the exact entropy dissipation.
Both schemes dissipate entropy almost identically.
Furthermore, in the bottom right panel of \cref{fig:euler_sod_conservation}, we show the residual entropy change as in \cref{eq:residuals}, which takes the entropy dissipation at shocks into account.
Although the FSBP operators lead to schemes that dissipate entropy, they do not match the entropy dissipation of the exact solution, and the error exceeds $10^{-4}$ in magnitude.
This is expected, as the schemes presented cannot dissipate entropy within elements but only at element boundaries, which is incorrect once shocks are present and may lead to entropy over- and underproduction.

\begin{figure}[tb]
  \centering
  \includegraphics{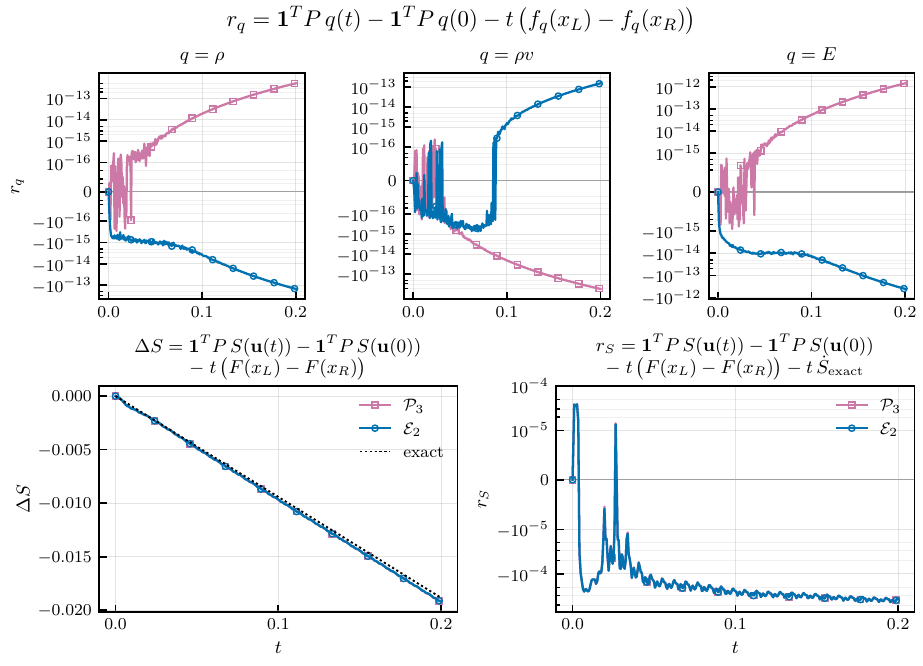}
    \caption{Time-integrated residuals \cref{eq:residuals} for the modified Sod shock tube with initial data \cref{eq:modified_sod}. The top row shows the residual $r_q$ of the discrete mass, momentum and total energy, that is, the change of the invariant corrected for the flux through the domain boundaries. The bottom left panel shows the change $\Delta S$ of the total entropy with only the boundary flux subtracted, against the entropy change $\sigma_S$ of the exact solution; the bottom right panel shows the entropy residual $r_S$, in which $\sigma_S$ is subtracted as well. We compare the FSBP operator for the space $\mathcal{E}_2 = \mathrm{span}\{1, x, e^x\}$ against the standard DGSEM polynomial SBP operator for $\mathcal{P}_3$, both on $64$ elements. All vertical axes except in the bottom-left panel use a symmetric logarithmic scale. We used the relaxation RK method for time stepping and entropy-stable surface fluxes.}
  \label{fig:euler_sod_conservation}
\end{figure}

\subsection{Two-dimensional compressible Euler equations}
\label{subsec:num_euler2d}

Finally, we briefly show that the FSBP operators and the flux-differencing framework extend to multiple dimensions using a tensor-product approach.
We consider the two-dimensional compressible Euler equations on a tensor-product Cartesian grid, where we apply the one-dimensional FSBP operators of \Cref{sec:FSBP} dimension by dimension.

\subsubsection*{Kelvin--Helmholtz instability}

As a more demanding test of nonlinear robustness, we consider the Kelvin--Helmholtz instability for the two-dimensional compressible Euler equations \cite{ruedaramirez2021subcell,glaubitz2025generalized}.
On the periodic domain $\Omega = [-1,1]^2$ with ratio of specific heats $\gamma = 7/5$, the initial condition is
\begin{equation}\label{eq:khi_ic}
  \rho = \tfrac{1}{2} + \tfrac{3}{4} B, \quad
  p = 1, \quad
  v_1 = \tfrac{1}{2}(B - 1), \quad
  v_2 = \tfrac{1}{10}\sin(2\pi x),
\end{equation}
with the smoothed shear profile $B(x,y) = \tanh(15 y + 7.5) - \tanh(15 y - 7.5)$.
The two shear layers roll up into vortices, and the flow becomes increasingly under-resolved, making this a demanding test of the scheme's robustness.
We integrate to the final time $t = 3$ with the entropy-stable flux and compare the equal-cost pair $\mathcal{P}_3$ and $\mathcal{E}_2$.
\Cref{fig:khi} shows the density at $t = 3$. 
Both operators remain stable up to that time.
Neither run uses shock capturing or a positivity limiter, and both operators fail just before $t \approx 3.6$ on this mesh.
Without flux-differencing and instead using the standard weak formulation of the DG method, the scheme crashes earlier; for $\mathcal{P}_3$ it crashes at $t \approx 2.25$ and for $\mathcal{E}_2$ at $t \approx 2.05$.

\begin{figure}[tb]
  \centering
  \includegraphics{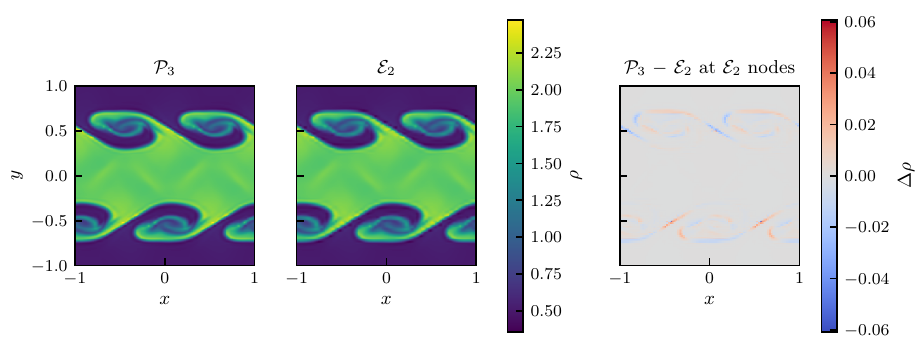}
    \caption{Density of numerical solutions to the Kelvin--Helmholtz instability with initial data \cref{eq:khi_ic} at $t = 3$ on $32 \times 32$ elements. We compare the standard DGSEM polynomial operator for $\mathcal{P}_3$ (left) against our FSBP operator for the space $\mathcal{E}_2 = \mathrm{span}\{1, x, e^x\}$ (middle), both with $4$ nodes per element and direction, using entropy-stable surface fluxes. The right panel shows their difference, evaluated at the nodes of the $\mathcal{E}_2$ operator.}
  \label{fig:khi}
\end{figure}
 
\section{Summary and outlook}
\label{sec:summary}

We have extended the flux-differencing framework from polynomial SBP operators to FSBP operators for general (potentially non-polynomial) function spaces.
Under two mild assumptions on the operator---a diagonal norm and the inclusion of the constants in the function space $\mathcal{F}$---the resulting semi-discretization is conservative in the sense of Lax--Wendroff (\cref{cor:conservation}), semi-discretely entropy-conservative on a single element (\cref{thm:entropy_cons}), and entropy-stable across element interfaces (\cref{thm:coupled}).
Complementing these structural properties, we have provided what is, to the best of our knowledge, the first systematic accuracy analysis of FSBP operators:
Their truncation error is governed by how well $\mathcal{F}$ simultaneously approximates a function and its derivative (\cref{lem:truncation}), and their design order is characterized by an algebraic Wronskian criterion on $\mathcal{F}$ (\cref{prop:design_order}), with suitable non-polynomial spaces attaining the same design order as polynomial spaces of equal dimension.
Moreover, the flux-differencing semi-discretization retains the design order of the underlying FSBP operator (\cref{prop:accuracy,cor:consistency}), while tailoring $\mathcal{F}$ to the solution can substantially reduce the error constant (\cref{rem:tailored_accuracy}).
The analysis is element-local and hence also covers element-wise varying function spaces and operators (\cref{cor:varying_spaces,rem:mixing_spaces}).
The construction is operator-independent: the entropy-conservative two-point flux is reused unchanged from the polynomial case, and only the underlying function space differs.
Our numerical experiments for the Burgers equation and the one- and two-dimensional compressible Euler equations confirm the discrete conservation and entropy properties and show that function spaces tailored to the solution---here exponential spaces---can be more accurate than polynomial bases of the same cost, in line with our accuracy analysis.

Several directions remain open.
On the theoretical side, extending the truncation-error and design-order analysis of \Cref{sec:accuracy} to a full convergence analysis of the flux-differencing semi-discretization remains an open problem.

Closely related is the gap between theory and practice for the design order itself.
\Cref{prop:design_order} guarantees the design order $\dim \mathcal{F} - 1$, whereas the Gaussian FSBP operators in \Cref{sec:numerics} converge one order faster than this suggests, namely at the same rate as polynomial SBP operators with the same number $N$ of nodes, which indicates the design order $N - 1$ (\cref{rem:eoc_vs_design_order}).

Throughout, we have restricted ourselves to closed operators, whose grids include the boundary nodes. Extending the framework to open FSBP operators, for which boundary data enter through a separate projection, is a natural next step.
Beyond the exponential spaces considered here, the construction applies verbatim to other solution-tailored spaces (trigonometric, radial-basis, or data-driven).
Furthermore, allowing the function space to vary is particularly promising: spatial variation between elements is already covered by the present element-local construction, whereas letting the space evolve in time---for instance, to follow a moving boundary layer or shock---additionally requires an entropy-stable projection between bases, which we leave for future work.

Finally, curved meshes to represent complex geometries is another avenue, and arguably the one most relevant for practitioners.
For polynomial SBP operators, conservative and entropy-stable flux differencing on curvilinear elements is well established \cite{gassner2016split,wintermeyer2017entropy,crean2018entropy,fernandez2019extension}, where the scheme is applied to the contravariant fluxes on a reference element, and conservation, free-stream preservation, and entropy stability rest on the SBP property together with discrete metric identities.
For FSBP operators, the metric terms would now be computed by $\mathcal{F}$-exact operators, which do not necessarily differentiate polynomial element mappings exactly, and a function space tailored in physical coordinates is transformed by the element mapping into an element-dependent space on the reference element.
Hence, either the tailoring has to be carried out in reference coordinates, or an operator has to be constructed for every element---forfeiting the translation invariance exploited in \Cref{sec:numerics}---for which the multi-dimensional constructions of \cite{glaubitz2023multi} offer a starting point.

\section*{Acknowledgements}
JG acknowledges support from the Swedish Research Council (VR) Starting Grant \#2025-05370, the Zenith Career Development Grant \#26.07, the Swedish e-Science Research Centre (SeRC), and the National Academic Infrastructure for Supercomputing in Sweden (NAISS) grants \#2025/22-1599 and \#2024/22-1207.

\section*{Tool disclosure}
Generative AI (Claude) assisted in preparing this work.  

\appendix 
\newpage
\section{Notation}
\label{app:notation}

For convenience, we summarize the notation used throughout the paper in \cref{tab:notation}.
The guiding convention is as follows.
An upright bold symbol such as $\mathbf{u}$ denotes a \emph{nodal vector}, i.e., the values of a scalar grid function at the $N$ grid points. 
In contrast, a bold italic symbol such as $\boldsymbol{u}$ denotes a \emph{vector-valued function} or, when evaluated at a point, a state in $\R^m$.
A subscript indexes the grid node (and, in the multi-element setting, the element).
The discrete solution of an $m$-component system is thus an $N \times m$ array $\mathbf{u}$ whose row $\boldsymbol{u}_i \in \R^m$ is the state at node $i$.
The matrix operators $D, P, Q, R, B$ act on the nodal index and componentwise across the $m$ fields.

\begin{table}[htbp]
\centering
\caption{Summary of the notation used throughout the paper.}
\label{tab:notation}
\renewcommand{\arraystretch}{1.25}
\begin{tabular}{@{}l p{0.64\textwidth}@{}}
\toprule
Symbol & Meaning \\
\midrule
\multicolumn{2}{@{}l}{\emph{Continuous quantities}} \\
$\boldsymbol{u}, \boldsymbol{f}$ & state $\boldsymbol{u} \colon \R \times \R_{\ge 0} \to \R^m$ and flux $\boldsymbol{f} \colon \R^m \to \R^m$ \\
$S, F$ & entropy and entropy flux, $S, F \colon \R^m \to \R$ (scalar valued) \\
$\boldsymbol{w}, \psi$ & entropy variables $\boldsymbol{w} = S'(\boldsymbol{u})^T \in \R^m$ and potential $\psi = \boldsymbol{w}^T \boldsymbol{f} - F$ \\
$\boldsymbol{f}^{\mathrm{vol}}, \boldsymbol{f}^{\mathrm{surf}}$ & two-point volume and surface fluxes, $\R^m \times \R^m \to \R^m$ \\
$\Lambda$ & dissipation matrix of the entropy-stable surface flux \\
\midrule
\multicolumn{2}{@{}l}{\emph{FSBP operators} (act on the nodal index, componentwise across fields)} \\
$D = P^{-1} Q$ & FSBP derivative operator in $\R^{N \times N}$ \\
$P$ & diagonal norm $\diag(p_1, \dots, p_N)$ with weights $p_i > 0$ \\
$Q$ & satisfies the SBP property $Q + Q^T = R^T B R$ \\
$R, B$ & restriction $R \in \R^{2 \times N}$ and boundary operator $B = \diag(-1, 1)$ \\
$\mathbf{1}$ & vector of ones $[1, \dots, 1]^T \in \R^N$ \\
\midrule
\multicolumn{2}{@{}l}{\emph{Discrete quantities on a single element}} \\
$\mathbf{u}$ & discrete solution (grid function), an array in $\R^{N \times m}$ \\
$\boldsymbol{u}_i$ & nodal state $\boldsymbol{u}(x_i) \in \R^m$ at node $i$ (row $i$) \\
$\mathbf{f}$ & nodal flux array, with row $i$ equal to $\boldsymbol{f}(\boldsymbol{u}_i)$ \\
$\vol$ & volume term, with row $\vol_i = \sum_n 2 D_{in}\, \boldsymbol{f}^{\mathrm{vol}}(\boldsymbol{u}_i, \boldsymbol{u}_n) \in \R^m$ \\
$\mathbf{f}^{\mathrm{surf}}$ & boundary surface fluxes $\boldsymbol{f}^{\mathrm{surf}}_L, \boldsymbol{f}^{\mathrm{surf}}_R \in \R^m$ \\
$\mathbf{w}$ & nodal entropy variables, with row $i$ equal to $\boldsymbol{w}_i = \boldsymbol{w}(\boldsymbol{u}_i)$ \\
$\bar{\boldsymbol{f}}_{i+\frac{1}{2}}$ & subcell (numerical) flux between nodes $i$ and $i+1$, in $\R^m$ \\
$\mathbf{1}^T P \mathbf{u}$ & discrete conserved quantity $\sum_i p_i \boldsymbol{u}_i \in \R^m$ \\
\midrule
\multicolumn{2}{@{}l}{\emph{Multi-element setting}} \\
$\mathbf{u}_k$ & discrete solution on element $\Omega_k$, an array in $\R^{N_k \times m}$ \\
$\boldsymbol{u}_{k,i}$ & state at node $i$ of element $\Omega_k$, in $\R^m$ \\
$\boldsymbol{u}^\pm_{k+\frac{1}{2}}, \boldsymbol{w}^\pm_{k+\frac{1}{2}}$ & left/right interface traces at the interface $x_{k+\frac{1}{2}}$ \\
\bottomrule
\end{tabular}
\end{table}
\section{Proof of \cref{cor:conservation}}
\label{app:conservation}

We start by showing that the flux-differencing volume term inherits the telescoping (subcell finite-volume) structure of the polynomial case \cite{fisher2013discretely,fisher2013high}.
This will allow us to prove \cref{cor:conservation} subsequently.

\begin{lemma}[Telescoping form]\label{lem:telescoping}
	Let $D = P^{-1}Q$ be a diagonal-norm $\mathcal{F}$-exact FSBP operator with $1 \in \mathcal{F}$ and let $\boldsymbol{f}^{\mathrm{vol}}$ be consistent and symmetric.
	Furthermore, let
	\begin{equation}\label{eq:subcell_flux}
		\bar{\boldsymbol{f}}_{i+\frac{1}{2}}
			:= \sum_{k=1}^{i} \sum_{l=i+1}^{N} 2 Q_{kl}\, \boldsymbol{f}^{\mathrm{vol}}(\boldsymbol{u}_k, \boldsymbol{u}_l),
			\quad i=1,\dots,N-1.
	\end{equation}
	Then, for every interior node $i \in \{2, \dots, N-1\}$, we have
	\begin{equation}\label{eq:telescoping}
		P_{ii}\, \vol_i = \bar{\boldsymbol{f}}_{i+\frac{1}{2}} - \bar{\boldsymbol{f}}_{i-\frac{1}{2}}.
	\end{equation}
	Moreover, the subcell fluxes \cref{eq:subcell_flux} are consistent, i.e., $\bar{\boldsymbol{f}}_{i+\frac{1}{2}} = \boldsymbol{f}(\boldsymbol{u})$ whenever $\boldsymbol{u}_1 = \dots = \boldsymbol{u}_N = \boldsymbol{u}$.
\end{lemma}

\begin{proof}
	Since $P$ is diagonal, $Q_{ij} = P_{ii} D_{ij}$, so that
	\begin{equation}\label{eq:tel_vol}
		P_{ii}\, \vol_i = \sum_{j=1}^{N} 2 Q_{ij}\, \boldsymbol{f}^{\mathrm{vol}}(\boldsymbol{u}_i, \boldsymbol{u}_j) .
	\end{equation}
	To evaluate the right-hand side of \cref{eq:telescoping}, we write the two subcell fluxes from \cref{eq:subcell_flux} as double sums over node pairs,
	\begin{equation}\label{eq:tel_two}
		\bar{\boldsymbol{f}}_{i+\frac{1}{2}}
			= \sum_{k=1}^{i} \sum_{l=i+1}^{N} 2 Q_{kl}\, \boldsymbol{f}^{\mathrm{vol}}(\boldsymbol{u}_k, \boldsymbol{u}_l), \quad
		\bar{\boldsymbol{f}}_{i-\frac{1}{2}}
			= \sum_{k=1}^{i-1} \sum_{l=i}^{N} 2 Q_{kl}\, \boldsymbol{f}^{\mathrm{vol}}(\boldsymbol{u}_k, \boldsymbol{u}_l).
	\end{equation}
	The two sums share all terms with $k \le i-1$ and $l \ge i+1$, which therefore cancel in the difference.
	The only surviving terms of $\bar{\boldsymbol{f}}_{i+\frac{1}{2}}$ are those in the row $k = i$ (with $l > i$), and the only surviving terms of $\bar{\boldsymbol{f}}_{i-\frac{1}{2}}$ are those in the column $l = i$ (with $k < i$).
	Hence,
	\begin{equation}\label{eq:tel_diff}
		\bar{\boldsymbol{f}}_{i+\frac{1}{2}} - \bar{\boldsymbol{f}}_{i-\frac{1}{2}}
			= \sum_{l > i} 2 Q_{il}\, \boldsymbol{f}^{\mathrm{vol}}(\boldsymbol{u}_i, \boldsymbol{u}_l)
			- \sum_{k < i} 2 Q_{ki}\, \boldsymbol{f}^{\mathrm{vol}}(\boldsymbol{u}_k, \boldsymbol{u}_i) .
	\end{equation}
	Now recall the off-diagonal identity $Q_{ki} = -Q_{ik}$ for $k \neq i$, which holds because $(Q + Q^T)_{ki} = (R^T B R)_{ki} = 0$ away from the boundary.
	In combination with the symmetry $\boldsymbol{f}^{\mathrm{vol}}(\boldsymbol{u}_k, \boldsymbol{u}_i) = \boldsymbol{f}^{\mathrm{vol}}(\boldsymbol{u}_i, \boldsymbol{u}_k)$, we rewrite \cref{eq:tel_diff} as
	\begin{equation}\label{eq:tel_sym}
		\bar{\boldsymbol{f}}_{i+\frac{1}{2}} - \bar{\boldsymbol{f}}_{i-\frac{1}{2}}
			= \sum_{l > i} 2 Q_{il}\, \boldsymbol{f}^{\mathrm{vol}}(\boldsymbol{u}_i, \boldsymbol{u}_l)
			+ \sum_{k < i} 2 Q_{ik}\, \boldsymbol{f}^{\mathrm{vol}}(\boldsymbol{u}_i, \boldsymbol{u}_k)
			= \sum_{j \neq i} 2 Q_{ij}\, \boldsymbol{f}^{\mathrm{vol}}(\boldsymbol{u}_i, \boldsymbol{u}_j) .
	\end{equation}
	The missing diagonal term vanishes, $2 Q_{ii}\, \boldsymbol{f}^{\mathrm{vol}}(\boldsymbol{u}_i, \boldsymbol{u}_i) = 0$, because $Q_{ii} = \tfrac{1}{2}(R^T B R)_{ii} = 0$ at the interior node $i$.
	Adding it to \cref{eq:tel_sym} and comparing with \cref{eq:tel_vol} yields \cref{eq:telescoping}.
	Finally, for a constant state $\boldsymbol{u}_1 = \dots = \boldsymbol{u}_N = \boldsymbol{u}$, consistency gives $\boldsymbol{f}^{\mathrm{vol}}(\boldsymbol{u}, \boldsymbol{u}) = \boldsymbol{f}(\boldsymbol{u})$, so that the double sum in \cref{eq:subcell_flux} factorizes:
	\begin{equation}\label{eq:tel_const}
		\bar{\boldsymbol{f}}_{i+\frac{1}{2}} = \sum_{k=1}^{i} \sum_{l=i+1}^{N} 2 Q_{kl}\, \boldsymbol{f}^{\mathrm{vol}}(\boldsymbol{u}, \boldsymbol{u}) = 2 \boldsymbol{f}(\boldsymbol{u}) \sum_{k \le i}\, \sum_{l > i} Q_{kl} .
	\end{equation}
	It thus remains to show that $\sum_{k \le i}\sum_{l > i} Q_{kl} = \tfrac{1}{2}$.
	Because $1 \in \mathcal{F}$, we have $Q\mathbf{1} = P D \mathbf{1} = 0$ (\cref{lem:constants}) and every row of $Q$ sums to zero.
	Hence $\sum_{l > i} Q_{kl} = -\sum_{l \le i} Q_{kl}$ for each $k$.
	Summing this over $k \le i$, symmetrizing the resulting double sum, and inserting $Q + Q^T = R^T B R$ gives
	\begin{equation}\label{eq:tel_sumQ}
		\begin{aligned}
			\sum_{k \le i} \sum_{l > i} Q_{kl}
				= - \sum_{k \le i} \sum_{l \le i} Q_{kl}
				= - \frac{1}{2} \sum_{k, l \le i} \big( Q_{kl} + Q_{lk} \big)
				= - \frac{1}{2} \sum_{k, l \le i} (R^T B R)_{kl}
				= - \frac{1}{2} (R^T B R)_{11}
				= \frac{1}{2} .
		\end{aligned}
	\end{equation}
	Here, the last two equalities use that, for the interior node $i \le N-1$, the index block $\{1, \dots, i\}^2$ contains the left boundary node $1$ but not the right boundary node $N$; since $R^T B R = \diag(-1, 0, \dots, 0, 1)$, its only nonzero entry in this block is $(R^T B R)_{11} = -1$.
	Substituting $\sum_{k \le i}\sum_{l > i} Q_{kl} = \tfrac{1}{2}$ into \cref{eq:tel_const} yields $\bar{\boldsymbol{f}}_{i+\frac{1}{2}} = 2 \boldsymbol{f}(\boldsymbol{u}) \cdot \tfrac{1}{2} = \boldsymbol{f}(\boldsymbol{u})$, as claimed.
\end{proof}

We are now in a position to show \cref{cor:conservation}.

\begin{proof}[Proof of \cref{cor:conservation}]
	Since $1 \in \mathcal{F}$, the constant function is reproduced at the boundary, so $R\mathbf{1} = [1, 1]^T$.
	Multiplying \cref{eq:fd_scheme} by $\mathbf{1}^T P$ and using $\mathbf{1}^T R^T = (R\mathbf{1})^T$ gives
	\begin{equation}
		\frac{\d}{\d t} \mathbf{1}^T P \mathbf{u}
			= -\mathbf{1}^T P\, \vol + (R\mathbf{1})^T B \left( R\mathbf{f} - \mathbf{f}^{\mathrm{surf}} \right) .
	\end{equation}
	For the volume term, $Q_{ij} = P_{ii} D_{ij}$ gives $\mathbf{1}^T P\, \vol = \sum_{i,j} 2 Q_{ij}\, \boldsymbol{f}^{\mathrm{vol}}(\boldsymbol{u}_i, \boldsymbol{u}_j)$.
	Symmetrizing in $(i,j)$ with the symmetry $\boldsymbol{f}^{\mathrm{vol}}(\boldsymbol{u}_i, \boldsymbol{u}_j) = \boldsymbol{f}^{\mathrm{vol}}(\boldsymbol{u}_j, \boldsymbol{u}_i)$ and using $Q + Q^T = R^T B R$ then yields the telescoping identity
	\begin{equation}
		\mathbf{1}^T P\, \vol
			= \sum_{i,j} (Q_{ij} + Q_{ji})\, \boldsymbol{f}^{\mathrm{vol}}(\boldsymbol{u}_i, \boldsymbol{u}_j)
			= \sum_{i,j} (R^T B R)_{ij}\, \boldsymbol{f}^{\mathrm{vol}}(\boldsymbol{u}_i, \boldsymbol{u}_j)
			= \boldsymbol{f}(\boldsymbol{u}_N) - \boldsymbol{f}(\boldsymbol{u}_1).
	\end{equation}
	The last equality uses $R^T B R = \diag(-1, 0, \dots, 0, 1)$ and the consistency $\boldsymbol{f}^{\mathrm{vol}}(\boldsymbol{u}_n, \boldsymbol{u}_n) = \boldsymbol{f}(\boldsymbol{u}_n)$.
	By the same diagonal structure, $(R\mathbf{1})^T B (R \mathbf{f}) = \boldsymbol{f}(\boldsymbol{u}_N) - \boldsymbol{f}(\boldsymbol{u}_1)$ and $(R\mathbf{1})^T B \mathbf{f}^{\mathrm{surf}} = \boldsymbol{f}^{\mathrm{surf}}_R - \boldsymbol{f}^{\mathrm{surf}}_L$.
	Substituting these three identities, we obtain
	\begin{equation}
		\frac{\d}{\d t} \mathbf{1}^T P \mathbf{u}
			= -\left( \boldsymbol{f}(\boldsymbol{u}_N) - \boldsymbol{f}(\boldsymbol{u}_1) \right)
			+ \left( \boldsymbol{f}(\boldsymbol{u}_N) - \boldsymbol{f}(\boldsymbol{u}_1) \right)
			- \left( \boldsymbol{f}^{\mathrm{surf}}_R - \boldsymbol{f}^{\mathrm{surf}}_L \right)
			= \boldsymbol{f}^{\mathrm{surf}}_L - \boldsymbol{f}^{\mathrm{surf}}_R ,
	\end{equation}
	which yields \cref{eq:conservation} and thus the assertion.
\end{proof}
\section{Proof of \cref{thm:entropy_cons}}
\label{app:entropy}

We first show a key algebraic identity: when contracted with the entropy variables, the volume term telescopes to the entropy flux.
This will allow us to subsequently prove \cref{thm:entropy_cons}.

\begin{lemma}[Volume entropy identity]\label{lem:vol_entropy}
	Let $D = P^{-1}Q$ be a diagonal-norm $\mathcal{F}$-exact FSBP operator with $1 \in \mathcal{F}$, let $\boldsymbol{f}^{\mathrm{vol}}$ be an EC flux satisfying \cref{eq:EC_condition}, and let $\mathbf{w}$ have rows $\boldsymbol{w}_i = \boldsymbol{w}(\boldsymbol{u}_i)$.
	Then
	\begin{equation}\label{eq:vol_entropy}
		\mathbf{w}^T P\, \vol = F(\boldsymbol{u}_N) - F(\boldsymbol{u}_1) ,
	\end{equation}
	where $\mathbf{w}^T P\, \vol = \sum_i P_{ii}\, \boldsymbol{w}_i^T \vol_i$.
\end{lemma}

\begin{proof}
	Using the diagonal norm and $Q_{ij} = P_{ii} D_{ij}$, we have
	\begin{equation}
		\mathbf{w}^T P\, \vol
			= \sum_{i,j} 2\, Q_{ij}\, \boldsymbol{w}_i^T \boldsymbol{f}^{\mathrm{vol}}(\boldsymbol{u}_i, \boldsymbol{u}_j) .
	\end{equation}
	Substituting $2 Q_{ij} = (R^T B R)_{ij} + (Q_{ij} - Q_{ji})$, which follows from the SBP property $Q + Q^T = R^T B R$, splits the sum into three parts:
	\begin{equation}
		\mathbf{w}^T P\, \vol
			= \sum_{i,j} (R^T B R)_{ij}\, \boldsymbol{w}_i^T \boldsymbol{f}^{\mathrm{vol}}(\boldsymbol{u}_i, \boldsymbol{u}_j)
			+ \sum_{i,j} Q_{ij}\, \boldsymbol{w}_i^T \boldsymbol{f}^{\mathrm{vol}}(\boldsymbol{u}_i, \boldsymbol{u}_j)
			- \sum_{i,j} Q_{ji}\, \boldsymbol{w}_i^T \boldsymbol{f}^{\mathrm{vol}}(\boldsymbol{u}_i, \boldsymbol{u}_j) .
	\end{equation}
	The first part is, by $R^T B R = \diag(-1, 0, \dots, 0, 1)$ and consistency,
	\begin{equation}
		\sum_{i,j} (R^T B R)_{ij}\, \boldsymbol{w}_i^T \boldsymbol{f}^{\mathrm{vol}}(\boldsymbol{u}_i, \boldsymbol{u}_j)
			= \boldsymbol{w}_N^T \boldsymbol{f}(\boldsymbol{u}_N) - \boldsymbol{w}_1^T \boldsymbol{f}(\boldsymbol{u}_1).
	\end{equation}
	We next relabel $i \leftrightarrow j$ and use the symmetry $\boldsymbol{f}^{\mathrm{vol}}(\boldsymbol{u}_j, \boldsymbol{u}_i) = \boldsymbol{f}^{\mathrm{vol}}(\boldsymbol{u}_i, \boldsymbol{u}_j)$, so that the second and third parts combine into
	\begin{equation}
		\sum_{i,j} Q_{ij} (\boldsymbol{w}_i - \boldsymbol{w}_j)^T \boldsymbol{f}^{\mathrm{vol}}(\boldsymbol{u}_i, \boldsymbol{u}_j) = \sum_{i,j} Q_{ij} (\psi_i - \psi_j) = \psi^T Q \mathbf{1} - \mathbf{1}^T Q \psi,
	\end{equation}
	where we used the EC condition \cref{eq:EC_condition} and collected the nodal potentials into $\psi = [\psi_1, \dots, \psi_N]^T$.
	By \cref{lem:constants}, $Q\mathbf{1} = P D \mathbf{1} = 0$, so $\psi^T Q\mathbf{1} = 0$.
	At the same time, $\mathbf{1}^T Q = \mathbf{1}^T (R^T B R - Q^T) = \mathbf{1}^T R^T B R$ gives $\mathbf{1}^T Q \psi = (R\mathbf{1})^T B (R\psi) = \psi_N - \psi_1$.
	Collecting terms and using $\psi_i = \boldsymbol{w}_i^T \boldsymbol{f}(\boldsymbol{u}_i) - F(\boldsymbol{u}_i)$ yields
	\begin{equation}
		\mathbf{w}^T P\, \vol
			= \left( \boldsymbol{w}_N^T \boldsymbol{f}(\boldsymbol{u}_N) - \psi_N \right)
				- \left( \boldsymbol{w}_1^T \boldsymbol{f}(\boldsymbol{u}_1) - \psi_1 \right)
			= F(\boldsymbol{u}_N) - F(\boldsymbol{u}_1)
	\end{equation}
	and thus the assertion.
\end{proof}

\Cref{lem:vol_entropy} is the discrete counterpart of $\int_{x_L}^{x_R} \boldsymbol{w}^T \partial_x \boldsymbol{f} \intd x = [F]_{x_L}^{x_R}$.
We can now show \cref{thm:entropy_cons}.

\begin{proof}[Proof of \cref{thm:entropy_cons}]
	Contracting \cref{eq:fd_scheme} with $P_{ii}\, \boldsymbol{w}_i^T$ and summing over $i$ gives
	\begin{equation}\label{eq:entropy_cons_proof1}
		\frac{\d}{\d t} \mathbf{1}^T P\, S(\mathbf{u})
			+ \mathbf{w}^T P\, \vol
			+ \sum_{i} \boldsymbol{w}_i^T \big( R^T B (\mathbf{f}^{\mathrm{surf}} - R\mathbf{f}) \big)_i
		= 0,
	\end{equation}
	where we used $\boldsymbol{w}_i^T \frac{\d \boldsymbol{u}_i}{\d t} = \frac{\d}{\d t} S(\boldsymbol{u}_i)$ and $\mathbf{w}^T P\, \vol = \sum_i P_{ii}\, \boldsymbol{w}_i^T \vol_i$.
	Since $R^T B (\mathbf{f}^{\mathrm{surf}} - R\mathbf{f})$ is supported at the two boundary nodes, the last sum reduces to
	\begin{equation}\label{eq:entropy_cons_proof2}
		\sum_{i} \boldsymbol{w}_i^T \big( R^T B (\mathbf{f}^{\mathrm{surf}} - R\mathbf{f}) \big)_i
			= \boldsymbol{w}_N^T \big( \boldsymbol{f}^{\mathrm{surf}}_R - \boldsymbol{f}(\boldsymbol{u}_N) \big)
			- \boldsymbol{w}_1^T \big( \boldsymbol{f}^{\mathrm{surf}}_L - \boldsymbol{f}(\boldsymbol{u}_1) \big).
	\end{equation}
	Inserting \cref{eq:vol_entropy,eq:entropy_cons_proof2} into \cref{eq:entropy_cons_proof1} and grouping the boundary contributions with $\psi_i = \boldsymbol{w}_i^T \boldsymbol{f}(\boldsymbol{u}_i) - F(\boldsymbol{u}_i)$ yields \cref{eq:entropy_balance} and thus the assertion.
\end{proof}

\bibliographystyle{siamplain}
\bibliography{references}

\end{document}